\documentclass{amsart}

\usepackage[utf8]{inputenc}

\usepackage{geometry}
\usepackage{amsthm}
\usepackage{amssymb}
\usepackage{amsmath}
\usepackage{comment}
\usepackage{thm-restate}
\usepackage{url} 
\usepackage{hyperref}
\usepackage[noabbrev,capitalise]{cleveref}

\usepackage{multicol}

\usepackage{enumitem}

\usepackage{color}
\usepackage[normalem]{ulem}

\usepackage{tikz}
\usetikzlibrary{snakes}

\newcommand{\eps}{\varepsilon}
\newcommand{\mc}[1]{\mathcal{#1}}

\renewcommand{\v}{\textup{\textsf{v}}}
\newcommand{\e}{\textup{\textsf{e}}}
\renewcommand{\d}{\textup{\textsf{d}}}

\newcommand{\s}[1]{\left(#1\right)}

\theoremstyle{plain}
\newtheorem{thm}{Theorem}[section]
\newtheorem{lem}[thm]{Lemma}

\newtheorem{cor}[thm]{Corollary}

\newtheorem{conj}[thm]{Conjecture}
\newtheorem{que}[thm]{Question}

\theoremstyle{definition}

\newcommand{\N}{\mathbb{N}}

\allowdisplaybreaks

\makeatletter
\@namedef{subjclassname@2020}{%
\textup{2020} Mathematics Subject Classification}
\makeatother

\begin{document}

\title{Limits of degeneracy for colouring graphs with forbidden minors}

\thanks{The authors are supported by the Natural Sciences and Engineering Research Council of Canada (NSERC). Les auteurs sont supportés par le Conseil de recherches en sciences naturelles et en génie du Canada (CRSNG)}


\author{Sergey Norin}
\address{Department of Mathematics and Statistics, McGill University, Montr\'eal, Canada}
\email{snorin@math.mcgill.ca}
\urladdr{www.math.mcgill.ca/snorin/}

\author{J\'er\'emie Turcotte}
\address{Department of Mathematics and Statistics, McGill University, Montr\'eal, Canada}
\email{mail@jeremieturcotte.com}
\urladdr{www.jeremieturcotte.com}
\begin{abstract}
	Motivated by Hadwiger's conjecture, Seymour asked which graphs $H$ have the property that every non-null graph $G$ with no $H$ minor has a vertex of degree at most $|V(H)|-2$. We show that for every monotone graph family $\mathcal{F}$ with strongly sublinear separators, all sufficiently large bipartite graphs $H \in \mathcal{F}$ with bounded maximum degree have this property. None of the conditions that $H$ belongs to $\mathcal{F}$, that $H$ is bipartite and that $H$ has bounded maximum degree can be omitted.	\end{abstract}

\maketitle

%%%%%%%%%%%%%%%%%%%%%%%%%%%%%%%%%%%%%%%%%%%%%%%%%%%%%
%%%%%%%%%%%%%%%%%%%%%%%%%%%%%%%%%%%%%%%%%%%%%%%%%%%%%
%%%%%%%%%%%%%%%%%%%%%%%%%%%%%%%%%%%%%%%%%%%%%%%%%%%%%
%%%%%%%%%%%%%%%%%%%%%%%%%%%%%%%%%%%%%%%%%%%%%%%%%%%%%
%%%%%%%%%%%%%%%%%%%%%%%%%%%%%%%%%%%%%%%%%%%%%%%%%%%%%
%%%%%%%%%%%%%%%%%%%%%%%%%%%%%%%%%%%%%%%%%%%%%%%%%%%%%
%%%%%%%%%%%%%%%%%%%%%%%%%%%%%%%%%%%%%%%%%%%%%%%%%%%%%

\section{Introduction}

In $1943$, Hadwiger proposed the following conjecture relating the chromatic number and complete minors.\footnote{All graphs in this paper are finite, simple and undirected. Given graph $H$ and $G$, we say that $H$ is \emph{a minor} of $G$ and write $H\preceq G$ if a graph isomorphic to $H$ can be obtained from a subgraph of $G$ by contracting edges. We denote the complete graph on $t$ vertices by $K_t$. A \emph{$k$-colouring} of a graph $G$ is a map $c: V(G) \to S$ for some set $S$ of colours with $|S|=k$ such that $c(u) \neq c(v)$ for every pair of adjacent $u,v \in V(G)$. A graph $G$ is \emph{$k$-colourable}  if it admits a $k$-colouring.}

\begin{conj}[Hadwiger's conjecture \cite{hadwiger_uber_1943}]\label{conj:hadwiger}
	For every positive integer $t$  every graph with no $K_{t}$ minor is $(t-1)$-colourable. 
\end{conj}

Hadwiger's conjecture is considered by many to be one of the most important open problems in graph theory. It is a notoriously difficult problem; in particular, it generalizes the Four Colour Theorem. Hadwiger's conjecture is only known to be true for $t\leq 6$.
The cases $t\leq 4$ were proved by Hadwiger \cite{hadwiger_uber_1943}. The case $t=5$ was proved by Appel and Haken as a consequence of their famous proof of the Four Colour Theorem \cite{appel_every_1977,appel_every_1977-1}. (The equivalence between these two statements was established earlier by Wagner \cite{wagner_uber_1937}.)  The case $t=6$ was proved by Robertson, Seymour and Thomas \cite{robertson_hadwigers_1993}, also using the Four Colour Theorem. See Seymour \cite{nash_hadwigers_2016} for a recent survey of  results and open problems related to Hadwiger's conjecture.

In the 1980s, Kostochka~\cite{kostochka_minimum_1982,kostochka_lower_1984} and Thomason~\cite{thomason_extremal_1984} proved that every graph with no $K_t$ minor is  $O(t\sqrt{\log{t}})$-colourable, and the order of magnitude of their upper bound remained unchanged until recently.
In fact, Kostochka and Thomason established a stronger result. They have shown that every graph $G$ with  $\delta(G) = \Omega(t\sqrt{\log{t}})$ has a $K_t$ minor, where we use $\delta(G)$ to denote the minimum degree of a graph $G$. Equivalently,  every non-null graph with no  $K_t$ minor has a vertex of degree $O(t\sqrt{\log{t}})$. A standard ``degeneracy'' inductive argument  implies that every  graph with no   $K_t$ minor is $O(t\sqrt{\log{t}})$-colourable. 
Note that the ``easy'' cases of Hadwiger's conjecture ($t \leq 4$) also follow via degeneracy, i.e. for $t \leq 4$ every non-null  graph $G$ with $\delta(G) \geq t-1$ has a $K_t$ minor. 

The Kostochka-Thomason bound on the order minimum degree sufficient to force a $K_t$ minor cannot be improved \cite{kostochka_minimum_1982,kostochka_lower_1984,fernandez_de_la_vega_maximum_1983} and  the possibility that colouring graphs with no $K_t$ minor requires $\Omega(t\sqrt{\log{t}})$ colours was left open, until recently, when the ``degeneracy barrier'' was broken by Postle, Song and the first author~\cite{norin_breaking_2023}. Even more recently, Delcourt and Postle \cite{delcourt_reducing_2022} have shown that  graphs with no $K_{t}$ minor are $O(t\log\log{t})$-colourable.

Given the apparent difficulty of Hadwiger's conjecture, many relaxations have been considered. We are interested in the following relaxation proposed by Seymour. Let $\v(H)$ denote the number of vertices of a graph $H$.

\begin{conj}[Seymour \cite{nash_hadwigers_2016,seymour_birs_2017}]\label{conj:h-hadwiger}
	For every graph $H$ with $\v(H) = t$, every graph with no $H$ minor is $(t-1)$-colourable.
\end{conj}

As $H$ is a subgraph of $K_t$, the validity of \cref{conj:h-hadwiger} for $H$ is implied by Hadwiger's conjecture for $t=\v(H)$. Note further that $K_{t-1}$ has no $H$ minor and is not $(t-2)$-colourable, and so  the number of colours in \cref{conj:h-hadwiger} is optimal for every $H$.  Woodall \cite{woodall_list_2001} and Seymour (in private communication) previously conjectured a more narrow weakening of Hadwiger's conjecture for $H=K_{s,t}$, where $K_{s,t}$ is the complete bipartite graph with parts of sizes respectively $s$ and $t$.

\cref{conj:h-hadwiger} is known to hold for some graphs $H$. 
Hendrey and Wood \cite{hendrey_extremal_2018} have proved it when $H$ is the Petersen graph. Lafferty and Song \cite{lafferty_every_2022,lafferty_every_2022-1} have shown that \cref{conj:h-hadwiger} holds for some graphs $H$ on $8$ and $9$ vertices; see therein for further background and references on \cref{conj:h-hadwiger} for small $H$. 
Denote by $K_{s,t}^*$ the graph obtained from the complete bipartite graph $K_{s,t}$ by making all the vertices in the part of size $s$  pairwise adjacent. Kostochka \cite{kostochka_kst_2010} has shown that \cref{conj:h-hadwiger} holds for $H=K_{s,t}^*$ when $t$ is sufficiently large compared to $s$, and later proved~\cite[Theorem 3]{kostochka_kst_2014}  that $t = \Omega((s\log s)^3)$ suffices. It follows that \cref{conj:h-hadwiger} holds whenever $H$ is bipartite and one of the parts of the bipartition is sufficiently large compared to the other. 

In~\cite{seymour_birs_2017}, Seymour asked, in particular, for which graphs $H$ does \cref{conj:h-hadwiger} follow from degeneracy.

\begin{que}[Seymour \cite{seymour_birs_2017}]\label{que:h-hadwiger}
For which graphs $H$ with $\v(H) = t$ does every non-null graph with no $H$ minor
have a vertex of degree at most $t-2$?
\end{que}

Let us say that $H$ is \emph{a Hadwiger-amenable graph} or \emph{an HA graph}, for brevity, if the answer to \cref{que:h-hadwiger} is positive.

As noted in \cite{seymour_birs_2017}, every graph $G$ contains as a subgraph every tree of order $\delta(G)+1$. It follows that trees are Hadwiger-amenable, but to the best of our knowledge no other general classes of HA graphs were previously known. Let us note that if $H$ is an HA graph then not only is every graph $G$ with no $H$ minor $(t-1)$-colourable, but is also $(t-1)$-list-colourable (and even the stronger $(t-1)$-DP-colourable), while many other methods establishing bounds for colouring are harder to extend to list-colouring. Hadwiger's conjecture, in particular, is known to be false for list colouring \cite{barat_disproof_2011} (see also \cite{steiner_improved_2022}).

Our main result states that every sufficiently large bipartite graph with bounded maximum degree and good separation properties is an HA graph. We will also show that none of these conditions can be entirely dismissed, hence providing a very rough characterization of large HA graphs. 

Before stating our main result more precisely, let us introduce necessary definitions and notation, some of which has been already mentioned above. We will use the notation $\N=\{1,2,\dots\}$ and $[n]=\{1,\dots,n\}$ (for $n\in \N$). Let $G$ be a graph. We write $\v(G)$ and $\e(G)$ for, respectively, the number of vertices and edges of $G$. If $S\subseteq V(G)$, then $G[S]$ will denote the subgraph of $G$ induced by $S$ and $G-S$ will denote the subgraph of $G$ induced by $V(G)\setminus S$. We denote by $\delta(G)$ and $\Delta(G)$ the minimum degree and maximum degree of $G$, respectively.

 %If $F\subseteq \binom{V(G)}{2}$, then $G+F$ is the graph on vertex set $V(G)$ and edge set $E(G)\cup F$. For $u\in V(G)$, we write $N_G(u)$ for neighbourhood of $u$, the set of vertices of $G$ adjacent to $u$. If $S\subseteq V(G)$, we write $N_G(S)=\bigcup_{u\in S}N_G(u)\setminus S$ for the neighbourhood of $S$. For $u\in V(G)$, we write $\deg_G(u)=|N_G(u)|$ for the degree of $u$ in $G$. 

A \emph{(proper) separation} of a graph $G$ is a pair of subsets $(A,B)$ of vertices of $G$ such that $A\cup B=V(G)$, $A\nsubseteq B$, $B\nsubseteq A$ and no edge of $G$ has one end in $A-B$ and the other in $B-A$. The \emph{order} of the separation $(A,B)$ is $|A\cap B|$. A separation is said to be \emph{balanced} if $|A|,|B|\leq \frac{2}{3}\v(G)$. A graph family $\mathcal F$ has \emph{strongly sublinear separators} if $\mathcal F$ is closed under taking subgraphs and there exists $c>0$ and $0<\beta<1$ such that every graph $G\in \mathcal F$ has a balanced separation of order at most $c\v(G)^\beta$.

We are  now ready state our main result.

\begin{restatable}{thm}{mainthm}\label{thm:main}
    For every graph family $\mathcal F$ with strongly sublinear separators and every $\Delta \in \N$, there exists $M=M_{\ref{thm:main}}(\mathcal F,\Delta)$ satisfying the following. If $H\in \mathcal F$ is a bipartite graph with $\Delta(H)\leq \Delta$ and $\v(H)\geq M$ then $H$ is Hadwiger-amenable. That is,
    if  $G$ is a non-null graph with $\delta(G) \geq \v(H)-1$, then $H$ is a minor of $G$.
\end{restatable}

In \cref{sec:tightness}, we show that none of the conditions of \cref{thm:main} can be omitted, while the rest of the paper is occupied by the proof \cref{thm:main}.
We  outline the proof of \cref{thm:main} in \cref{sec:outline}, and derive it from a number of technical results, which are proved in Sections~\ref{sec:small}--\ref{sec:minorfrompieces}.

%%%%%%%%%%%%%%%%%%%%%%%%%%%%%%%%%%%%%%%%%%%%%%%%%%%%%
%%%%%%%%%%%%%%%%%%%%%%%%%%%%%%%%%%%%%%%%%%%%%%%%%%%%%
%%%%%%%%%%%%%%%%%%%%%%%%%%%%%%%%%%%%%%%%%%%%%%%%%%%%%
%%%%%%%%%%%%%%%%%%%%%%%%%%%%%%%%%%%%%%%%%%%%%%%%%%%%%
%%%%%%%%%%%%%%%%%%%%%%%%%%%%%%%%%%%%%%%%%%%%%%%%%%%%%
%%%%%%%%%%%%%%%%%%%%%%%%%%%%%%%%%%%%%%%%%%%%%%%%%%%%%
%%%%%%%%%%%%%%%%%%%%%%%%%%%%%%%%%%%%%%%%%%%%%%%%%%%%%

\section{Proof outline}\label{sec:outline}

In this section, we present the tools used in the proof of \cref{thm:main} and outline the proof.

\subsection{Tools}

\subsubsection*{Models}

We will often certify that a graph $H$ is a minor of a graph $G$ by exhibiting a model of $H$ in $G$. A \emph{model} $\mu$ of $H$ in $G$ assigns to every vertex of $v \in V(H)$ a set $\mu(v)$ of vertices of $G$ such that
\begin{itemize}
	\item $\mu(u) \cap \mu(v) = \emptyset$ for every pair of  distinct $u,v \in V(H)$, 
	\item $G[\mu(v)]$ is connected for every $v \in V(H)$, and
	\item for every edge $uv 
	\in E(H)$ there exist $u' \in \mu(u)$ and $v' \in \mu(v)$ such that $u'v' \in E(G)$.
\end{itemize} 

For $U \subseteq V(H)$, let $\mu(U)=\bigcup_{v \in U}\mu(v)$ for brevity. Note that the last two conditions in the above definition of a model can be replaced by the following single condition.
\begin{itemize}
	\item $G[\mu(U)]$ is connected for every $U \subseteq V(G)$ such that $H[U]$ is connected.
\end{itemize} 

The following properties of models are well-known and not difficult to verify.

\begin{lem}\label{l:model} If $H$ and $G$ are graphs, then
  	\begin{enumerate}[label=\normalfont(\alph*)]
  		\item $H \preceq G$ if and only if there exists a model of $H$ in $G$, and
  		\item if $F \subseteq E(H)$, $\mu$ is a model of $H - F$ in $G$ and $\{P_f\}_{f \in F}$ is a collection of paths in $G$ such that
  			\begin{itemize}
  				\item  $P_{uv}$ has one end in $\mu(u)$, the other in $\mu(v)$, and is otherwise disjoint from $\mu(V(H))$ for every $uv \in F$, and
				\item the paths $\{P_f\}_{f \in F}$ are pairwise internally vertex-disjoint,
 			\end{itemize}  
 			then $H \preceq G$.
  	\end{enumerate}
\end{lem}

\subsubsection*{Density}

While \cref{thm:main} is concerned with the minimum degree condition necessary to guarantee existence of given graph as a minor, the average degree conditions have been much more thoroughly investigated and provide a starting point for our argument. 

We define the \emph{density} of a non-null graph $G$ as $\d(G)=\frac{\e(G)}{\v(G)}$. We note that $\d(G)$ is half the average degree of $G$, and so $\frac{\delta(G)}{2}\leq\d(G)$. The \emph{extremal function} of a graph $H$, denoted by $c(H)$, as the supremum of densities of $H$-minor-free graphs, i.e. $H \preceq G$ for every non-null graph $G$ with $\d(G)>c(H)$.

One of our main tools is the following result obtained independently by Haslegrave, Kim and Liu~\cite{haslegrave_extremal_2022}\footnote{Haslegrave, Kim and Liu~\cite{haslegrave_extremal_2022} state their theorem for proper minor-closed families, but their proof holds for families with strongly sublinear separators.}, and by Hendrey, Norin and Wood~\cite{hendrey_extremal_2022}. 

\begin{thm}[{\cite[Theorem 2.2]{haslegrave_extremal_2022}, \cite[Theorem 1.1]{hendrey_extremal_2022}}]\label{thm:bipartiteextremalfunction}
	For every graph family $\mathcal F$ with strongly sublinear separators and every $\varepsilon>0$, there exists $M=M_{\ref{thm:bipartiteextremalfunction}}(\mathcal F,\varepsilon)$ such that  $$c(H) \leq (1+\eps)\frac{\v(H)}{2}$$ for every bipartite graph $H\in \mathcal F$ and $\v(H)\geq M$.
\end{thm}

\cref{thm:bipartiteextremalfunction} implies that a slight weakening of \cref{thm:main} holds with  the condition $\delta(G) \geq \v(H)-1$ replaced by $\delta(G) \geq (1+\eps)\v(H)$. 
At several points in the proof, we will be able to 
replace $G$ or a subgraph of $G$ by a  denser minor and apply \cref{thm:bipartiteextremalfunction}.

\subsubsection*{Using strongly sublinear separators}
For $s\in \N$, let $\mathcal F_s$ be the class of all graphs with component size (that is, the number of vertices in the component) at most $s$. Then, it is clear that $\mathcal F_s$ is a family with strongly sublinear separators. We will use this to apply \cref{thm:bipartiteextremalfunction} to graphs with bounded component size.

In addition to \cref{thm:bipartiteextremalfunction} we use several other technical results from~\cite{hendrey_extremal_2022}. The following result states that graphs in families with strongly sublinear separators are in fact always close to graphs with bounded component size. It is essentially~\cite[Lemma 7.3]{hendrey_extremal_2022}, but as the statement is slightly different we include its proof in \cref{app:compsize} for completeness.

\begin{restatable}{lem}{compsize}\label{lem:compsize}
	For every graph family $\mathcal F$ with strongly sublinear separators and every $\delta > 0$ there exists $s = s_{\ref{lem:compsize}}(\mathcal F,\delta)$ such that for any graph $H \in \mathcal F$ there exists a graph $H'$ and  $F\subseteq E(H')$ such that \begin{itemize}
		\item $H \preceq H'$,
	    \item $\v(H)\leq\v(H')\leq(1+\delta)\v(H)$,
	    \item $\Delta(H')\leq \Delta(H)+2$,
	    \item $|F|\leq \delta \v(H)$, 
	    \item for every component $J$ of $H'-F$ we have $\v(J) \leq s$ and $J$ is isomorphic to a subgraph of $H$, and
	    \item no edge of $F$ has both ends in the same component of $H'-F$. 
	\end{itemize}
\end{restatable}

Note that in \cref{lem:compsize}, $H'$ is not necessarily bipartite even if $H$ is bipartite. However, $H'-F$ is bipartite if $H$ is bipartite. By \cref{lem:compsize}, in the proof of \cref{thm:main} it suffices to show that $H' \preceq G$ for $H'$ satisfying the conditions of the lemma. We do so by building a model of $H'-F$ and then, given some appropriate collection of paths, extending it to all of $H'$ using \cref{l:model}(b). 

\subsubsection*{Connectivity}

To extend a given model we need the graph $H$ to satisfy certain connectivity assumptions. The actual assumptions that are both usable and attainable are somewhat technical. Let us now introduce them.

Let $d,k \geq 0$, let   $G$ be a graph and let $X \subsetneq V(G)$.  
We say that a pair $(G,X)$ is \emph{$(d,k)$-dense} if \begin{itemize}
	\item $\deg_G(v) \geq d$ for every $v \in V(G)-X$,
	\item for every separation $(A,B)$ of $G$ of order at most $k$ we have $A\setminus B \subseteq X$ or $B\setminus A \subseteq X$.
\end{itemize}
For brevity, if $d=0$ we can omit the first condition, that is we write that $(G,X)$ is a \emph{$k$-dense} pair.

\begin{lem}\label{l:denseSub}
If $k > 0$ and $G$ is a non-null graph, then there exists a subgraph $G'$ of $G$ and $X \subsetneq V(G')$ such that $|X| \leq 2k$ and $(G',X)$ is $(\delta(G),k)$-dense.
\end{lem}

\begin{proof} If $G$ admits no separation of order less than $2k$ then $(G,\emptyset)$ is $(\delta(G),k)$-dense. Thus we assume that there exists a separation $(A,B)$ of $G$ of order less than $2k$ and choose such a separation with $A$ minimal. 
	
	We claim that $(G[A],A \cap B)$ is 	$(\delta(G),k)$-dense. Clearly $\deg_{G[A]}(v) = \deg_G(v) \geq \delta(G)$ for every $v \in V(G[A])\setminus (A\cap B)$, and so we assume for a contradiction that there exists a separation $(A',B')$ of $G[A]$ with $|A' \cap B'| \leq k$ such that neither of $(A'\setminus B')\setminus (A \cap B)=A'\setminus (B\cup B')$ and $(B'\setminus A') \setminus (A \cap B)=B'\setminus (A'\cup B)$ is empty. As 
	$$|B \cap (A'\setminus B')| + |B \cap (B'\setminus A')| \leq |A \cap B| < 2k,$$
	we assume without loss of generality that $|B \cap (A'\setminus B')|<k$. Then, $(A',B \cup B')$ is a separation of $G$ of order
	$$|A' \cap (B \cup B')| = |A' \cap B'| + |B \cap (A'\setminus B')|< 2k. $$
	
	Furthermore, $A' \subsetneq A$ given that $(A',B')$ is a separation of $G[A]$. Thus  $(A',B \cup B')$ contradicts the choice of $(A,B)$. This contradiction implies the claim and thus the lemma. 
\end{proof}	

By \cref{l:denseSub}, we can replace the graph $G$ in \cref{thm:main} by a $(\v(H)-1, \eps \v(H))$-dense pair $(G',X)$ with $|X| \leq 2\eps\v(H)$ for any $\varepsilon>0$.

In one of the cases of the proof of our main result, the collections of paths needed to extend models as discussed in the previous subsection will be found by using the following lemma, which is a slight modification of \cite[Lemma 6.5]{hendrey_extremal_2022} that involves our non-standard connectivity condition. The proof of \cite[Lemma 6.5]{hendrey_extremal_2022} translates to the setting we need with trivial changes\footnote{In the proof of \cite[Lemma 6.5]{hendrey_extremal_2022}, one uses Menger's theorem \cite{menger_zur_1927} and the fact that the connectivity of the graph is at least $\varepsilon\v(G)$ to find at least $\varepsilon\v(G)$ paths between a pair of non-adjacent vertices $u,v$. Here, instead of using connectivity, for such $u,v\notin X$ we use the $\varepsilon\v(G)$-density to obtain that there is no vertex-cut separating $u$ and $v$ with order at most $\varepsilon\v(G)$, and again apply Menger's theorem. The rest of the proof is identical.}.

\begin{lem}[{\cite[Lemma 6.5]{hendrey_extremal_2022}}] 
	\label{l:linkage} 
	For every $\eps>0$, there exists $ \delta=\delta_{\ref{l:linkage}}(\eps)>0$ satisfying the following. If $(G,X)$ is an $\eps\v(G)$-dense pair, then there exists $Z \subseteq V(G)$ with $|Z| \leq \eps \v(G)$ such that for all $p_1,q_1,\dots p_t,q_t\in V(G)\setminus (X\cup Z)$ with $t \leq \delta\,\v(G)$, there exist a collection of pairwise internally vertex disjoint paths $P_1,\ldots, P_t$ in $G$ such that $P_i$ has ends $p_i$ and $q_i$ and $V(P_i) \setminus \{p_i,q_i\} \subseteq Z$ for every $1 \leq i \leq t$.
\end{lem}

In another case of the proof, we will instead use the following result allowing us to build a minor from pieces, the proof of which also uses this idea of obtaining the model by finding appropriate paths between models of subgraphs. We will prove this theorem in \cref{sec:minorfrompieces}.

\begin{restatable}{thm}{minorfrompieces}\label{t:minorfrompieces} 
	There exists $C = C_{\ref{t:minorfrompieces}}\in \N$ satisfying the following. Let $H$ be a graph, let $F \subseteq E(H)$ be such that $H - F$ is a disjoint union of graphs $H_1,\ldots,H_k$ and such that no edge of $F$ has both ends in the same component of $H-F$. If $G$ is a graph and $G_1,\ldots,G_k$ are pairwise vertex-disjoint subgraphs of $G$ such that  
	\begin{itemize}
		\item $\d(G_i) \geq C(c(H_i)+|F|)$ for every $i \in [k]$, and
		\item $\left(G,V(G) \setminus \bigcup_{i=1}^{k}V(G_i)\right)$ is $2|F|$-dense,
	\end{itemize}	  
	then $H \preceq G$. 
\end{restatable}

\subsection{Proof outline}

We will prove the following more technical version of \cref{thm:main}.

\begin{thm}\label{thm:newmain}
	For every $\Delta \in \N$, there exists $\alpha = \alpha_{\ref{thm:newmain}}(\Delta)>0$ such that for every graph family $\mathcal F$ with strongly sublinear separators and every $\eps > 0$ there exists $M=M_{\ref{thm:newmain}}(\mc{F},\Delta,\eps)$ such that the following holds. If $H\in \mathcal F$ is a bipartite graph with $\Delta(H)\leq \Delta$ and $\v(H) \geq M$ and $(G,X)$ is a $(\v(H)-1,\eps\v(H))$-dense pair such that $|X| \leq \alpha\v(H)$, then $H \preceq G$.
\end{thm}

\cref{thm:main} follows directly from \cref{thm:newmain} by setting $M_{\ref{thm:main}}(\mathcal F,\Delta)=M_{\ref{thm:newmain}}\left(\mathcal F,\Delta,\frac{\alpha_{\ref{thm:newmain}}(\mc{F},\Delta)}{2}\right)$ and applying \cref{l:denseSub} to $G$ with $k=\frac{\alpha_{\ref{thm:newmain}}(\mc{F},\Delta)}{2}\v(H)$ to obtain $(G',X)$ as noted above.

The proof of \cref{thm:newmain} is separated into three cases. When $\v(G)$ is only slightly larger then $\v(H)$ we use the following lemma, which we prove in \cref{sec:small}, to find $H$ as a subgraph of $G$.

\begin{restatable}{lem}{small}\label{lem:small} 
	Let $\Delta \in \N$, $H$ be a bipartite graph with $\Delta(H) \leq \Delta$, $G$ be a graph and $X \subsetneq V(G)$ be such that $\deg_G(v) \geq \v(H)-1$ for every $v \in V(G)\setminus X$. If $$\v(G) \leq \s{1+ \frac{1}{4\Delta(\Delta+1)}}\v(H)-1$$ and $$|X| \leq \frac{\v(H)}{(\Delta+1)(\Delta^2+1)},$$ then $H$ is isomorphic to a subgraph of $G$.
\end{restatable}

When $\v(G)$ is somewhat larger but still within a constant factor of $\v(H)$ we use the following result, which we prove in \cref{sec:bounded}, in combination with  Lemmas~\ref{lem:compsize} and~\ref{l:linkage}. 

\begin{restatable}{thm}{medium}\label{thm:boundedcomponentminor}
		For every $\Delta\in \N$, $\nu>0$, there exists $\mu=\mu_{\ref{thm:boundedcomponentminor}}(\Delta,\nu)>0$ such that for every $s\in \N$ there exists $M=M_{\ref{thm:boundedcomponentminor}}(\Delta,\nu,s)$ such that if $H$ is a bipartite graph with $\Delta(H)\leq \Delta$, maximum component size $s$ and $\v(H)\geq M$, and $G$ is a graph such that $\delta(G)\geq (1-\mu)\v(H)$ and $\v(G)\geq (1+\nu)\v(H)$, then $H$ is a minor of $G$.
\end{restatable}

Finally, when $\v(G)$ is much larger that $\v(H)$, we use the following density increment lemma in combination with \cref{lem:compsize} and Theorem~\ref{t:minorfrompieces}. 

\begin{restatable}{thm}{larger}\label{thm:repeatdensityincrement}
	There exists $\varepsilon=\varepsilon_{\ref{thm:repeatdensityincrement}}>0$ such that for every $K\in \N$ there exist $\varepsilon'=\varepsilon'_{\ref{thm:repeatdensityincrement}}(K)$ and $C=C_{\ref{thm:repeatdensityincrement}}(K)\geq 1$ such that for every $D \geq C$ and every $G$ such that $\delta(G)\geq D$ and $\v(G)\geq CD$, then either
	\begin{enumerate}
		\item $G$ contains vertex-disjoint subgraphs $J_1,\dots,J_K$ such that $\v(J_i)\leq \frac{D}{\varepsilon}$ and $\d(J_i)\geq \varepsilon D$ for every $i\in [K]$, or
		\item $G$ contains a minor $H$ such that $\d(H)\geq (1+\varepsilon')\frac{D}{2}$.
	\end{enumerate}
\end{restatable}

We are now ready to derive \cref{thm:newmain} from the above-mentioned results.

\begin{proof}[Proof of \cref{thm:newmain}] 

We begin by introducing the necessary parameters. Let $\beta=\frac{\eps_{\ref{thm:repeatdensityincrement}}}{4C_{\ref{t:minorfrompieces}}}$ and $K=\lceil 2/\beta \rceil$.

Given $\Delta$, let $\nu=\frac{1}{8(\Delta+1)(\Delta^2+1)}$ and  $\alpha=\min\left(\frac{\nu}{3},\frac{\mu_{\ref{thm:boundedcomponentminor}}(\Delta+2,\nu)}{3},\frac{\eps'_{\ref{thm:repeatdensityincrement}}(K)}{3(1+\eps'_{\ref{thm:repeatdensityincrement}}(K))}\right)$. Let $\gamma=\frac{(1+\eps'_{\ref{thm:repeatdensityincrement}}(K))(1-2\alpha)-1}{2}$; note that $\gamma>0$ by choice of $\alpha$.

Given $\mathcal F,\eps$, let $\eps'=\min\left(\frac{\varepsilon}{C_{\ref{thm:repeatdensityincrement}}(K)},\frac{\alpha}{C_{\ref{thm:repeatdensityincrement}}(K)}\right)$, $\delta =\min\s{\frac{\nu}{3(1+\nu)}, \frac{\alpha}{2(1-3\alpha)}, \delta_{\ref{l:linkage}}(\eps'),\frac{\eps}{2},\beta}$, $s=s_{\ref{lem:compsize}}(\mc{F},\delta)$ and finally $$M=\max\left( \frac{1}{\frac{1}{4\Delta(\Delta+1)}-\frac{1}{4(\Delta+1)(\Delta^2+1)}},\frac{2}{\alpha},M_{\ref{thm:boundedcomponentminor}}(\Delta+2,\nu,s),\frac{Ks}{K\beta-1-\delta},\frac{C_{\ref{thm:repeatdensityincrement}}(K)}{1-2\alpha},M_{\ref{thm:bipartiteextremalfunction}}(\mathcal F,\gamma),K(M_{\ref{thm:bipartiteextremalfunction}}(\mathcal F_s,1)+s)\right).$$

Given $H$ and $(G,X)$ as in the statement, we show that $H\preceq G$. As mentioned earlier, we divide the proof into three cases depending on the ratio between $\v(G)$ and $\v(H)$.
	
\vskip 5pt	
\noindent{\bf Case 1:} $ \v(G) \leq  (1 + 2\nu) \v(H)$.

Given that
$$\v(G)\leq (1+2\nu)\v(H)= \left(1+\frac{1}{4(\Delta+1)(\Delta^2+1)}\right)\v(H)\leq \s{1+ \frac{1}{4\Delta(\Delta+1)}}\v(H)-1,$$
where the last inequality follows by choice of $M$, and 
$$|X| \leq \alpha\v(H) \leq \nu \v(H) \leq\frac{\v(H)}{(\Delta+1)(\Delta^2+1)},$$
the conditions of \cref{lem:small} are satisfied. Hence, $H$ is isomorphic to a subgraph of $G$, and so $H\preceq G$.

\vskip 5pt	
For the remaining two cases, apply \cref{lem:compsize} with the above choice of $\delta$ to $H$ to obtain $H'$ and $F \subseteq E(H')$ such that $H \preceq H'$, $\v(H)\leq\v(H')\leq(1+\delta)\v(H)$, $\Delta(H')\leq \Delta(H)+2\leq \Delta+2$, $|F|\leq \delta \v(H)$, for every component $J$ of $H'-F$ we have $\v(J) \leq s$ and $J$ is isomorphic to a subgraph of $H$ (in particular, $H'-F$ is bipartite), and no edge of $F$ has both ends in the same component of $H'-F$.

\vskip 5pt	
\noindent{\bf Case 2:} $(1 + 2\nu)\v(H) \leq  \v(G) \leq C_{\ref{thm:repeatdensityincrement}}(K)\v(H)$.

Given that $(G,X)$ is $\varepsilon \v(H)$-dense and $\varepsilon'\v(G)\leq \eps' C_{\ref{thm:repeatdensityincrement}}(K)\v(H)\leq \varepsilon\v(H)$, it is also $\varepsilon' \v(G)$-dense. Apply \cref{l:linkage} with $\varepsilon'$ instead of $\varepsilon$ to $(G,X)$ to obtain $Z$ satisfying the conditions of the lemma. In particular,
$$|Z|  \leq \varepsilon'\v(G)\leq \eps' C_{\ref{thm:repeatdensityincrement}}(K)\v(H) \leq \alpha \v(H).$$

Let $G'=G-(X\cup Z)$. Then $|X\cup Z|\leq |X|+|Z| \leq 2\alpha \v(H) \leq \frac{2}{3}\nu \v(H)$ and so
$$\v(G') =\v(G)-|X\cup Z|\geq \s{1 + \frac{4}{3}\nu} \v(H) \geq  (1+\nu)(1+\delta)\v(H)\geq(1 + \nu)\v(H'),$$
and
\begin{align*}
	\delta(G')&\geq \delta(G)-|X\cup Z| \geq \s{1-2\alpha}\v(H)-1
	\geq\left(1-\frac{5}{2}\alpha\right)\v(H)\geq(1-3\alpha)(1+\delta)\v(H)\\
	&\geq (1-\mu_{\ref{thm:boundedcomponentminor}}(\Delta+2,\nu))\v(H'),
\end{align*}
where the fourth inequality follows by choice of $M$. Furthermore, $\v(H'-F)\geq \v(H)\geq M\geq M_{\ref{thm:boundedcomponentminor}}(\Delta+2,\nu,s)$.
 
Thus by \cref{thm:boundedcomponentminor} applied to the graph $H' - F$ in place of $H$, and $G'$ in place of $G$, there exists a model $\mu$ of $H'- F$ in $G'$. As $|F| \leq \delta \v(H) \leq \delta_{\ref{l:linkage}}(\eps')\v(G)$ and $Z$ was chosen to satisfy the conditions of \cref{l:linkage}, there exist pairwise internally vertex disjoint paths $\{P_{uv}\}_{uv \in F}$ in $G$, such that for every edge $uv \in F$ the path $P_{uv}$ has one end in $\mu(u)$ the other in $\mu(v)$ and is otherwise disjoint from $V(G')$ and so from $\mu(V(H))$. Thus by \cref{l:model}(b) there exists a model of $H'$ in $G$, and so $H \preceq H' \preceq G,$ as desired. 
 
\vskip 5pt 
\noindent{\bf Case 3:} $C_{\ref{thm:repeatdensityincrement}}(K)\v(H) \leq \v(G)$.

Let $H_1,\ldots,H_K$ be disjoint subgraphs of $H'-F$ such that for every $i\in [K]$, $H_i$ is a union of connected components of $H'-F$ and $\v(H_i) \leq \beta\v(H)$, and then subject to these conditions $\sum_{i=1}^K \v(H_i)$ is maximal, and subject to this condition $\max_{i,j\in [K]}|\v(H_i)-\v(H_j)|$ is minimal.

We first claim that $\v(H')=\sum_{i=1}^K \v(H_i)$. Suppose otherwise for a contradiction that this is not the case, that is at least one component $J$ of $H'-F$ is not in any of the $H_1,\dots,H_K$. Given that $\v(J)\geq s$, we know that for every $i\in [K]$, $\v(H_i)> \beta\v(H) - s$ (otherwise $J$ could be added to $H_i$). This implies that $(1+\delta)\v(H) \geq  \v(H')> \sum_{i=1}^K \v(H_i)> K(\beta\v(H) - s)$, from which it follows that $\v(H) < \frac{Ks}{K\beta-1-\delta}$ since $K\beta-1-\delta \geq 1 - \delta > 0$. This is a contradiction to the choice of $M$. This finishes the proof of the claim.

Let $G'=G-X$ and $D=(1-2\alpha)\v(H)$. Note that $D\geq C_{\ref{thm:repeatdensityincrement}}(K)$ by choice of $M$. We have that

$$\v(G')=\v(G)-|X|\geq C_{\ref{thm:repeatdensityincrement}}(K)\v(H)-\alpha\v(H)\geq C_{\ref{thm:repeatdensityincrement}}(K)(1-2\alpha)\v(H)=C_{\ref{thm:repeatdensityincrement}}(K)D$$
and
$$\delta(G')\geq \delta(G)-|X|\geq \v(H)-1-\alpha\v(H)\geq (1-2\alpha)\v(H)=D,$$
where the last inequality follows by choice of $M$.

Hence by \cref{thm:repeatdensityincrement} applied to $G'$ either
\begin{enumerate}
	\item $G'$ contains vertex-disjoint subgraphs $G_1,\dots,G_K$ such that  $\d(G_i)\geq \eps_{\ref{thm:repeatdensityincrement}}(1-2\alpha)\v(H)$ for every $i\in [K]$, or
	\item $G'$ contains a minor $G''$ such that $\d(G'')\geq (1+\eps'_{\ref{thm:repeatdensityincrement}}(K))(1-2\alpha)\frac{\v(H)}{2}>(1+\gamma)\frac{\v(H)}{2}$.
\end{enumerate}

Given that $\v(H)\geq M\geq M_{\ref{thm:bipartiteextremalfunction}}(\mathcal F,\gamma)$, \cref{thm:bipartiteextremalfunction} yields that $c(H)\leq (1+\gamma)\frac{\v(H)}{2}$. Thus if (2) holds we have $H \preceq G'' \preceq G' \preceq G$.

Suppose on the other hand note that (1) holds. By the last condition in the choice of the subgraphs, for every $i,j\in [K]$ we have $|\v(H_i)-\v(H_j)|\leq s$, since every component of $H'-F$ has order at most $s$ (if this is not the case for some pair, transfer a component from the largest of the two subgraphs to the smallest). As the average order of the $H_i$ is $\frac{\v(H')}{K}\geq \frac{\v(H)}{K}$, this implies that $\v(H_i)\geq \frac{\v(H)}{K}-s\geq \frac{M}{K}-s\geq M_{\ref{thm:bipartiteextremalfunction}}(\mathcal F_s,1)$ for every $i\in [K]$. Given that $H'-F$ has bounded component size $s$, $H_1,\dots,H_K\in \mathcal F_s$. Hence, \cref{thm:bipartiteextremalfunction} yields that $c(H_i)\leq (1+1)\frac{\v(H_i)}{2}=\v(H_i)$ for every $i\in [K]$. Then
$$\d(G_i) \geq \eps_{\ref{thm:repeatdensityincrement}}(1-2\alpha)\v(H) \geq  2C_{\ref{t:minorfrompieces}}\beta\v(H) \geq C_{\ref{t:minorfrompieces}}(\v(H_i)+\delta \v(H)) \geq C_{\ref{t:minorfrompieces}}(c(H_i)+|F|)$$
for every $i\in[K]$. Furthermore, $(G,X)$ is $\eps\v(H)$-dense and so $\left(G,V(G) \setminus \bigcup_{i=1}^{K}V(G_i)\right)$ is $|F|$-dense given that $2|F|\leq 2\delta\v(H)\leq \varepsilon\v(H)$ and $X\subseteq V(G) \setminus \bigcup_{i=1}^{K}V(G_i)$. Thus by \cref{t:minorfrompieces} (applied to $H'$ instead of $H$), $H \preceq H' \preceq G$, finishing the proof in this case.
\end{proof}	

Let us highlight why, using our methods, we must split the proof in three cases. The obstacles are broadly as follows.
\begin{itemize}
	\item \cref{lem:small} is only applicable when $\v(G)$ is very close to $\v(H)$. 
	\item \cref{thm:boundedcomponentminor} is not applicable when $\v(G)$ is too close to $\v(H)$.
	\item \cref{l:linkage} requires an $\varepsilon'\v(G)$-dense pair, however under our hypothesis we only have an $\varepsilon\v(H)$-dense pair, so we cannot use this lemma if $\v(G)$ is arbitrarily large compared to than $\v(H)$.
	\item Our use of \cref{t:minorfrompieces} requires us to find a large number $K$ of pieces with good density, which we obtain using \cref{thm:repeatdensityincrement}. The latter requires that $G$ be much larger than $D$; in the proof, $D$ is only slightly larger than $\v(H)$.
\end{itemize}

%%%%%%%%%%%%%%%%%%%%%%%%%%%%%%%%%%%%%%%%%%%%%%%%%%%%%
%%%%%%%%%%%%%%%%%%%%%%%%%%%%%%%%%%%%%%%%%%%%%%%%%%%%%
%%%%%%%%%%%%%%%%%%%%%%%%%%%%%%%%%%%%%%%%%%%%%%%%%%%%%
%%%%%%%%%%%%%%%%%%%%%%%%%%%%%%%%%%%%%%%%%%%%%%%%%%%%%
%%%%%%%%%%%%%%%%%%%%%%%%%%%%%%%%%%%%%%%%%%%%%%%%%%%%%
%%%%%%%%%%%%%%%%%%%%%%%%%%%%%%%%%%%%%%%%%%%%%%%%%%%%%
%%%%%%%%%%%%%%%%%%%%%%%%%%%%%%%%%%%%%%%%%%%%%%%%%%%%%

\section{Small case}\label{sec:small}

In this section, we prove \cref{lem:small}, which we restate for convenience.

\small*

\begin{proof}
	If follows from the fact that at least one vertex of $G$ has degree at least $\v(H)-1$ that $$\v(G)\geq \v(H)\geq (\Delta+1)(\Delta^2+1)|X|>(\Delta+1)|X|.$$
	
	Let $Y=V(G)\setminus X$, and if possible let $X_0 \subseteq X$ be chosen maximal such that $|N_G(X_0)\cap Y|< \Delta|X_0|$. In particular, $$|N_G(X_0)\cap Y|< \Delta|X_0|\leq \Delta|X| < \v(G)-|X| = |Y|.$$
	Otherwise, let $X_0=\emptyset$. In both cases, this implies there exists $v \in Y$ with no neighbours in $X_0$. and so $\deg_G(v) \leq \v(G- X_0) -1$. It then follows from the lower bound on $\deg_G(v)$ from the statement that
	\begin{equation*}
		\v(G- X_0) \geq \v(H).
	\end{equation*}
	Then, let $X'=X\setminus X_0$.
	
	We construct an injective homomorphism $\phi:V(H) \to V(G\setminus X_0)$, that is an injection such that $\phi(u)\phi(v) \in E(G)$ for every  $uv \in E(H)$. The existence of such an injection implies the lemma.
	
	We first choose $(A,B)$ a bipartition of $H$ such that $|A| \leq |B|$ and subject to this $|B|$ is minimum. We claim that $|A| \geq \frac{\v(H)}{\Delta+1}$. If not, then $|B|\geq \v(H)-|A|>\left(1-\frac{1}{\Delta+1}\right)\v(H)>|A|$ and in particular $B$ contains no isolated vertices by our choice of the bipartition. Thus $\Delta|A| \geq \e(H) \geq |B| = \v(H)-|A|$, implying the desired inequality.
	
	We then choose a set of vertices in $A$ to be mapped to $X'$ as follows. Let $A_0 \subseteq A$ be chosen so that $|A_0|=|X'|$ and no two vertices of $A_0$ have a common neighbour. Such a choice is possible. Indeed, otherwise then there exists $A_0  \subseteq A$  such that $|A_0| < |X'| \leq |X|$ and every vertex in $A-A_0$ shares a neighbour with a vertex in $A_0$. As at most $\Delta^2|A_0|$ vertices of $A\setminus A_0$ share neighbours with vertices in $A_0$, we have 
	\begin{align*}
		\frac{\v(H)}{\Delta+1} -|X| \leq |A|-|X| \leq |A\setminus A_0| \leq  \Delta^2|A_0| < \Delta^2 |X|,
	\end{align*}
	implying $|X| > \frac{\v(H)}{(\Delta+1)(\Delta^2+1)}$, a contradiction. Define $\phi|_{A_0}$ to be an arbitrary bijection $A_0 \to X'$.  
	
	By the choice of $X_0$ we have $|N_G(S)\cap Y| \geq \Delta|S|$ for every $S \subseteq X'$, and so by Hall's theorem (applied $\Delta$ times) there exist pairwise disjoint sets $(Y_v)_{v \in X'}$ such that  $Y_v \subseteq N_G(v)\cap Y$ and $|Y_v| = \Delta$ for every $v \in X'$. Hence, we can extend the injection $\phi$ to $B$ so that $\phi(N_{G}(v)\cap B) \subseteq Y_v \subseteq N_G(v)\cap Y$ for every $v \in A_0$.
	 
	 Let $A' = A\setminus A_0$. It remains to define $\phi$ on $A'$. Let $Y' = Y\setminus \phi(B)$. For every $v \in A'$ let
	 $$Y_v = \bigcap_{w \in N_H(v)\cap B} N_G(\phi(w))\cap Y'.$$
	 In other words, $Y_v$ is the set of possible choices for extending $\phi$ to $v$ in $Y'$ for $\phi$ to still be a homomorphism. In order for $\phi$ to also be an injection, our goal is then find a set of distinct representatives of the set system $(Y_v)_{v \in A'}$.
	 
	By Hall's theorem such a system of representatives exists as long as 
  	\begin{equation}\label{e:Hall}
  	\left|\bigcup_{v \in S}Y_v\right| \geq |S| 
  	\end{equation}
  	for every $S \subseteq A'$.
  	
  	We thus finish the proof by verifying that Hall's condition holds. First note that
  	$$|Y'|=|Y|-|\phi(B)|=(\v(G- X_0)-|X'|)-|B|\geq \v(H)-|B|-|X'|=|A|-|X'|.$$
  	
  	First suppose $|S| \leq \frac{|A|}{2}$. We can rewrite $Y_v = Y' \setminus\left( \bigcup_{w \in N_G(v)\cap B} (Y' \setminus N_{G}(\phi(w)))\right)$ for every $u \in A'$. Thus
  	\begin{align*}
  	|Y_v|
  	&\geq |Y'| - \sum_{w \in N_H(v)\cap B}|Y' \setminus N_{G}(\phi(w))| 
  	\geq |A|-|X'| - \sum_{w \in N_H(v)\cap B}|V(G) \setminus N_G(\phi(w))| \\
  	&\geq |A|-|X| - \sum_{w \in N_H(v)\cap B}(\v(G)-\deg_G(\phi(w)))
  	\geq |A|-\frac{\v(H)}{4(\Delta+1)} - \Delta(\v(G) - \v(H)+1) \\
  	&\geq |A|-\frac{\v(H)}{2(\Delta+1)}
  	\geq \frac{|A|}{2}\\
  	&\geq |S|,
  	 \end{align*}                                                                                                        
  	 as desired.
  	 
  	 We now suppose that $|S|>\frac{|A|}{2}$. We claim that $\bigcup_{v \in S}Y_v = Y'$ and so that $\left|\bigcup_{v \in S}Y_v\right|  = |Y'| \geq |A|-|X|=|A|-|A_0|=|A'| \geq |S|$, as desired. It thus only remains to establish this claim.
  	 
  	 Suppose for a contradiction it does not hold, that is there exists $S \subseteq A'$ (with $|S| > |A|/2$) such that $u \in Y'$ but for which $u \not \in Y_v$ for every $v \in S$. Let $B' \subseteq B$ be the set of $w \in B$ such that $\phi(w) \in Y\setminus N_G(u)$. Then $|B'| \leq |Y\setminus N_G(u)| \leq \v(G)-\deg(u) \leq \v(G) - \v(H)+1$. 
  	 
  	 Note that every $v \in S$ has  a neighbour in $B'$ by the choice of $S$ and $u$. It follows that
  	 $$|S| \leq |N_H(B')\cap A| \leq \Delta|B'| \leq \Delta(\v(G) - \v(H)+1) \leq \frac{\v(H)}{4(\Delta+1)} \leq \frac{|A|}{4},$$ 
  	 a contradiction. This finishes the proof of the claim and thus of the lemma.
\end{proof}

%%%%%%%%%%%%%%%%%%%%%%%%%%%%%%%%%%%%%%%%%%%%%%%%%%%%%
%%%%%%%%%%%%%%%%%%%%%%%%%%%%%%%%%%%%%%%%%%%%%%%%%%%%%
%%%%%%%%%%%%%%%%%%%%%%%%%%%%%%%%%%%%%%%%%%%%%%%%%%%%%
%%%%%%%%%%%%%%%%%%%%%%%%%%%%%%%%%%%%%%%%%%%%%%%%%%%%%
%%%%%%%%%%%%%%%%%%%%%%%%%%%%%%%%%%%%%%%%%%%%%%%%%%%%%
%%%%%%%%%%%%%%%%%%%%%%%%%%%%%%%%%%%%%%%%%%%%%%%%%%%%%
%%%%%%%%%%%%%%%%%%%%%%%%%%%%%%%%%%%%%%%%%%%%%%%%%%%%%

\section{Minors with bounded component size}\label{sec:bounded}
In this section, we prove \cref{thm:boundedcomponentminor}, which allows us to construct large bipartite minors with bounded component size.

Given a graph $H$, we say a graph $G$ is $H$-free if $G$ does not contain a copy of $H$, that is a subgraph isomorphic to $H$. We now cite a density increment result of Krivelevich and Sudakov \cite{krivelevich_minors_2009} for $K_{s,s}$-free graphs; note that a slightly weaker result of  Kühn and Osthus \cite[Theorem 11]{kuhn_complete_2004} would be sufficient for our purposes.

\begin{thm}[{\cite[Theorem 4.5]{krivelevich_minors_2009}}]\label{thm:kssfreedensityincrement}
	For every $s\in \N^{\geq 2}$, there exists $C=C_{\ref{thm:kssfreedensityincrement}}(s)>0$ such that if a graph $G$ is $K_{s,s}$-free, then it contains a minor $J$ such that $\d(J)\geq C\cdot (\d(G))^{1+\frac{1}{2(s-1)}}$.
\end{thm}

We are now ready to prove \cref{thm:boundedcomponentminor}, which we restate for convenience.

\medium*

\begin{proof}
    We begin by introducing the necessary parameters. We are given $\Delta,\nu$. Let $\varepsilon>0$ small enough such that
    \vspace*{-4mm}
    \begin{multicols}{2}
    \begin{enumerate}[label=(E\arabic*)]
        \item \label{C1} $\varepsilon<\min\left(\frac{1}{10},\nu\right)$
        \item \label{C3} $\frac{1-\varepsilon^2-4\varepsilon}{1-2\varepsilon^2}\geq 1-\frac{1}{3\Delta}$,
        \item \label{C4} $1 - 6\eps -\frac{12\eps}{\nu} \geq \frac{3}{4}$,
        \item \label{C5} $(1-2\varepsilon)(1+\nu)-(1-\varepsilon)>0$,\\
        \item \label{C6} $\frac{\nu}{1+\frac{\nu}{2}}>8\varepsilon$,
        \item \label{C7} $\left(1-\frac{2\varepsilon(2+\nu)}{1-2\varepsilon^2}\right)\geq \frac{3}{4}$,
        \item \label{C8} $ \frac{(1+\nu+2\varepsilon^2)}{1+\frac{\nu}{2}} > (1+2\varepsilon)$.
    \end{enumerate}
    \end{multicols}
    \vspace*{-7mm}\noindent It is easy to verify that this is possible. Choose $\mu=\varepsilon^2$.
    
    We are given $s$; we assume $s\geq 2$ without loss of generality. Let $M_1= \frac{M_{\ref{thm:bipartiteextremalfunction}}(\mathcal F_s,\varepsilon^2)}{2\varepsilon^2}$, $M_2$ large enough that $\frac{C_{\ref{thm:kssfreedensityincrement}}(s)}{2}\left(\frac{\varepsilon^2M_2}{2}\right)^\frac{1}{2(s-1)}> (1+\varepsilon^2)$, $M_3=\frac{2s}{\varepsilon^2}$ and finally $M=\max(M_1,M_2,M_3)$.
    
    Since $\v(H)\geq M_1> M_{\ref{thm:bipartiteextremalfunction}}(\mathcal F_s,\varepsilon^2)$, by Corollary \ref{thm:bipartiteextremalfunction} we have that either $G$ contains $H$ as a minor or $\d(G)\leq c(H)\leq (1+\varepsilon^2)\frac{\v(H)}{2}$. We may assume the latter case since in the first case the lemma holds. On the other hand, we have that $\d(G)\geq \frac{\delta(G)}{2}\geq (1-\varepsilon^2)\frac{\v(H)}{2}$.
    
    Let $G_0$ be an induced subgraph of $G$ with $\v(G_0)$ maximum such that
    \begin{enumerate}[label=(S\arabic*)]
        \item \label{stopsize} $\v(G_0)\leq (1-2\varepsilon^2)\v(H)$,
        \item \label{stopminor} $G_0$ contains a spanning subgraph isomorphic to $H_0$, where $H_0$ is a union of connected components of $H$,
    \end{enumerate}
    and, subject to the above,
    \begin{enumerate}
        \item if $\v(G)\geq (2+\nu)\v(H)$, select $G_0$ minimizing $\sum_{v\in V(G_0)} \deg_G(v)$, and
        \item if $(1+\nu)\v(H)\leq \v(G)< (2+\nu)\v(H)$, select $G_0$ maximizing $\sum_{v\in V(G_0)} \deg_{G_0}(v)$.
    \end{enumerate}
    This is always possible, given that the empty graph respects \ref{stopsize} and \ref{stopminor}.
    
    Denote $G'=G - V(G_0)$. Our goal is to show that $G'$ contains $H'=H-V(H_0)$ as a minor. If follows from \ref{stopminor} that $\v(G_0)= \v(H_0)$. Condition \ref{stopsize} then implies that $$\v(H')=\v(H)-\v(H_0)\geq 2\varepsilon^2\v(H).$$
    
    Since $\v(G_0)= \v(H_0)$, every vertex in $G'$ has at most $\v(H_0)$ neighbours outside of $G'$, hence $$\delta(G') \geq (1-\varepsilon^2)\v(H)-\v(H_0)= (1-\varepsilon^2)\v(H)-(\v(H)-\v(H'))=\v(H')-\varepsilon^2\v(H)$$ and $\d(G')\geq\frac{\v(H')-\varepsilon^2\v(H)}{2}$.
    
    First, consider the case $(1-2\varepsilon^2)\v(H)-s \geq \v(G_0)$. Then, $G'$ does not contain subgraphs isomorphic to any of the components of $H'$, as otherwise we could add the corresponding induced subgraph of $G'$ to $G_0$ without violating either \ref{stopminor} or  \ref{stopsize}, and so contradicting the choice of $G_0$.  Since $H$ is bipartite and has component size at most $s$, its components are subgraphs of $K_{s,s}$. Therefore $G'$ has no subgraph isomorphic to $K_{s,s}$. By Theorem \ref{thm:kssfreedensityincrement}, $G'$ contains a minor $J$ such that 
    \begin{align*}
        \d(J) &\geq C_{\ref{thm:kssfreedensityincrement}}(s) \cdot (\d(G'))^{1+\frac{1}{2(s-1)}} \\
        &\geq C_{\ref{thm:kssfreedensityincrement}}(s)\left(\frac{\v(H')-\varepsilon^2\v(H)}{2}\right)^{1+\frac{1}{2(s-1)}}\\
        &\geq C_{\ref{thm:kssfreedensityincrement}}(s)\left(\frac{(2\varepsilon^2-\varepsilon^2)\v(H)}{2}\right)^\frac{1}{2(s-1)}\left(\frac{\left(1-\frac{\varepsilon^2}{2\varepsilon^2}\right)\v(H')}{2}\right)\\
        &> (1+\varepsilon^2)\frac{\v(H')}{2}
    \end{align*}
    by choice of $M_2$.
    Since $\v(H')\geq 2\varepsilon^2 \v(H)\geq 2\varepsilon^2 M_1= M_{\ref{thm:bipartiteextremalfunction}}(\mathcal F_s,\varepsilon^2)$, by  \cref{thm:bipartiteextremalfunction} we have that $J$ contains $H'$ as a minor. Hence, $G'$ contains $H'$ as a minor, as desired.
    
   It remains to consider the case $(1-2\varepsilon^2)\v(H)-s \leq \v(G_0)\leq (1-2\varepsilon^2)\v(H)$. Note that this implies that $2\varepsilon^2\v(H)\leq\v(H')\leq 2\varepsilon^2\v(H)+s$ and that $$\v(G_0)\geq (1-2\varepsilon^2)\v(H)-s\geq (1-2\varepsilon^2)M_3-s=(1-2\varepsilon^2)\frac{2s}{\varepsilon^2}-s\geq s.$$ If $\d(G')>(1+\varepsilon^2)\frac{\v(H')}{2}$, then as in the previous paragraph $G'$ contains $H'$ as a minor by  \cref{thm:bipartiteextremalfunction}. Hence we assume $\d(G')\leq (1+\varepsilon^2) \frac{\v(H')}{2}$; we show this leads to a contradiction.
    
    Denote by $A$ the set of vertices of $G$ with degree smaller than $(1+2\varepsilon)\v(H)$. We claim that $|A| \geq (1-\eps) \v(G)$. Indeed, as $\delta(G) \geq (1-\varepsilon^2)\v(H)$, otherwise we would have
    $$\d(G)\geq \frac{|A|(1-\varepsilon^2)\v(H)+(\v(G)-|A|)\cdot(1+2\varepsilon)\v(H)}{2\v(G)} \geq (1+\varepsilon^2+\varepsilon^3)\frac{\v(H)}{2}>(1+\varepsilon^2)\frac{\v(H)}{2},$$
    contradicting our earlier assumption. 
 
    Denote by $B$ the set of vertices of $v \in V(G')$ such that $\deg_{G'}(v) < 4\varepsilon \v(H)$. Similarly to the above, we show that $|B| \geq (1-\eps)\v(G')$. Indeed, otherwise we would have 
    $$\d(G')\geq \frac{(\v(G')-|B|)\cdot 4\varepsilon \v(H)}{2\v(G')} \geq 2\varepsilon^2\v(H)\geq\v(H') -s \geq \left(1-\frac{s}{2\varepsilon^2\v(H)}\right)\v(H')>(1+\varepsilon^2)\frac{\v(H')}{2}$$ which is again a contradiction (the last inequality uses our choice of $M_3$).  Note that each vertex $v\in B$ has at least
    $$\deg_G(u)-\deg_{G'}(u)\geq(1-\varepsilon^2-4\varepsilon)\geq (1-5\varepsilon)\v(H) \geq \frac{(1-5\varepsilon)\v(H)}{(1-2\varepsilon^2)\v(H)}\v(G_0) \geq \s{1-\frac{1}{3\Delta}}\v(G_0)$$ neighbours in $G_0$, using \ref{C3}.
    
    Another consequence is that there are at least $(1-5\varepsilon)\v(H)\cdot (1-\varepsilon)\v(G') \geq (1-6\eps)\v(H)\v(G') $ edges between the vertices of $G'$ (specifically, the vertices in $B$) and the vertices of $G_0$. Let $C$ the set of vertices of $G_0$ with at least $\frac{\v(G')}{1+\frac{\nu}{2}}$ neighbours in $G'$. Then the number of edges between vertices of $G'$ and $G_0$ is upper bounded by
    $|C|\cdot \v(G')+ (\v(G_0)-|C|)\cdot \frac{\v(G')}{1+\frac{\nu}{2}}$, implying
    $$ |C| + \frac{\v(G_0)-|C|}{{1+\frac{\nu}{2}}} \geq (1-6\eps)\v(G_0),$$ and therefore
    \begin{align*}
        |C| \geq \s{1 - 6\eps -\frac{12\eps}{\nu}}\v(G_0) \geq \frac{3}{4}\v(G_0) 
    \end{align*}
    where the last inequality  uses \ref{C4}. 
    
    Let $A'=A\cap V(G_0)$ if $(1+\nu)\v(H)\leq \v(G)< (2+\nu)\v(H)$ and $A'=V(G_0)$ otherwise. We claim that $|C \cap A'|>\frac{1}{2}\v(G_0)$.
    
    If $A'=V(G_0)$, clearly $|C\cap A'|=|C|\geq \frac{3}{4}\v(G_0)>\frac{1}{2}\v(G_0)$. Otherwise,  $\v(G)< (2+\nu)\v(H)$, and     \begin{align*}
        |A'|
        &\geq|A|-\v(G')
        \geq (1-\varepsilon)\v(G)-\v(G')
        = \v(G_0)-\varepsilon\v(G)
        > \v(G_0)-\varepsilon(2+\nu)\v(H)\\
        &\geq\v(G_0)-\varepsilon(2+\nu)\frac{\v(G_0)+s}{1-2\varepsilon^2}
        \geq\left(1-\frac{2\varepsilon(2+\nu)}{1-2\varepsilon^2}\right)\v(G_0)\\
        &\geq \frac{3}{4}\v(G_0)
    \end{align*}
    using \ref{C7} and the bound $\v(G_0)\geq s$. We then have $$|C\cap A'|= |C|+|A'|-|C\cup A'|> \frac{3}{4}\v(G_0)+\frac{3}{4}\v(G_0)-\v(G_0)= \frac{1}{2}\v(G_0),$$
    and the claim also holds in this case. 
    
    As $|A|\geq (1-\varepsilon)\v(G)$ and $|B| \geq  (1-\eps)\v(G')$, we have 
    \begin{align*}
        |A\cap B|&= |(A \setminus V(G_0)) \cap B |
        \geq (|A| - \v(G_0)) - (\v(G')-|B|) \\ &\geq  (1-\eps)\v(G) - \v(G_0) - \eps  \v(G')
        =(1-2\varepsilon)\v(G)-(1-\varepsilon)\v(G_0)\\
        &\geq(1-2\varepsilon)(1+\nu)\v(H)-(1-\varepsilon)\v(G_0)
        \geq((1-2\varepsilon)(1+\nu)-(1-\varepsilon))\v(G_0) \\
        &>0
    \end{align*}
    by \ref{C5}. Choose $x\in A\cap B$. 
    
  	Our goal is to show that there exists  $y\in C \cap A'$ such that $N_{H_0}(y)\subseteq N_{G}(x)$ (considering $H_0$ as a subgraph of $G$ with some fixed embedding by \ref{stopminor}).
  	
  	Suppose this is not the case. Then, for each $y\in C\cap A'$, $y$ is adjacent (in $H$) to at least one vertex of $V(G_0)\setminus N_{G}(x)$. Hence, the total degree (in $H$) over vertices of $V(G_0)\setminus N_{G}(x)$ is at least $|C\cap A'|> \frac{1}{2}\v(G_0)$. By the Pigeonhole principle, at least one vertex $z\in V(G_0)\setminus N_{G}(x)$ has degree at least $\frac{\frac{1}{2}\v(G_0)}{|V(G_0)\setminus N_{G}(x)|}$ in $H$. Being in $B$, $x$ is adjacent to a proportion of at least $1-\frac{1}{3\Delta}$ of the vertices of $G_0$, which means that $\deg_{H_0}(z)\geq \frac{\frac{1}{2}\v(G_0)}{\frac{1}{3\Delta}\v(G_0)}>\Delta$. This is a contradiction to the maximum degree of $H$. Hence, such a $y$ exists.
  	
  	Removing $y$ of $G_0$ and instead adding $x$ still yields $H_0$ as a subgraph of $G$: keep the same embedding of $H_0$ in $G_0$ but replace $y$ by $x$, which is possible since $N_{H_0}(y)\subseteq N_{G}(x)$.
    
    Consider the case $\v(G)\geq (2+\nu)\v(H)$. In this case, we claim that $x$ has smaller degree in $G$ than $y$, which is the contradiction we are looking for. Indeed, being in $A$, the degree of $x$ is at most $(1+2\varepsilon)\v(H)$. On the other hand, since $y\in C$, its degree is at least
    \begin{align*}
       \frac{\v(G')}{1+\frac{\nu}{2}}=\frac{\v(G)-\v(G_0)}{1+\frac{\nu}{2}}
        \geq \frac{(2+\nu)\v(H)-(1-2\varepsilon^2)\v(H)}{1+\frac{\nu}{2}}
        = \frac{(1+\nu+2\varepsilon^2)}{1+\frac{\nu}{2}}\v(H)
        > (1+2\varepsilon)\v(H),
    \end{align*}
    the last inequality holding by \ref{C8}. This contradicts our initial choice of $G_0$.
    
    Now, consider the case $(1+\nu)\v(H)\leq \v(G)< (2+\nu)\v(H)$. In this case, we claim $x$ has greater degree in $G_0$ than $y$, which will contradict the choice of $G_0$. Since $y\in A'\subseteq A$, $\deg_G(y)\leq (1+2\varepsilon)\v(H)$, and since $y\in C$, it is adjacent to at least $\frac{1}{1+\frac{\nu}{2}}\v(G')$ vertices of $G'$. Hence, $y$ is adjacent to at most $(1+2\varepsilon)\v(H)-\frac{1}{1+\frac{\nu}{2}}\v(G')$ vertices of $G_0$. On the other hand, being in $B$, $x$ is adjacent to at least $(1-5\varepsilon)\v(H)-1$ vertices of $G_0$ excluding $y$. It then suffices to show that
    $$(1-5\varepsilon)\v(H)-1>(1+2\varepsilon)\v(H)-\frac{1}{1+\frac{\nu}{2}}\v(G').$$
    Using that $\v(G')=\v(G)-\v(G_0)$, we can rewrite the desired inequality as
    $$\frac{1}{1+\frac{\nu}{2}}\v(G)>1+7\varepsilon\v(H)+\frac{1}{1+\frac{\nu}{2}}\v(G_0).$$
  	Using that $\v(G_0)\leq (1-2\varepsilon^2)\v(H)<\v(H)$, $\v(G)\geq (1+\nu)\v(H)$ and $\v(H)\geq M_3\geq\frac{2s}{\varepsilon^2}\geq \frac{1}{\varepsilon}$, it suffices to show that
    $$\frac{1+\nu}{1+\frac{\nu}{2}}>8\varepsilon+\frac{1}{1+\frac{\nu}{2}}.$$
    This is a direct consequence of \ref{C6}. This is the contradiction we are looking for, since replacing $y$ by $x$ would give more edges inside $G_0$.

    In all cases, the union of the $H'$ minor in $G'$ and the $H_0$ subgraph of $G_0$ gives us an $H$ minor in $G$.
\end{proof}

We note that for our purposes, it would be sufficient to prove this result for graphs $G$ such that $\v(G)\leq C_0\v(H)$, for some large $C_0$. In fact, the previous proof could be modified so that the cutoff point between the two cases is $C_0\v(H)$ instead of $(2+\nu)\v(H)$, and so only one case would be required in the proof. However, we believe the more general result is interesting in its own right.

%%%%%%%%%%%%%%%%%%%%%%%%%%%%%%%%%%%%%%%%%%%%%%%%%%%%%
%%%%%%%%%%%%%%%%%%%%%%%%%%%%%%%%%%%%%%%%%%%%%%%%%%%%%
%%%%%%%%%%%%%%%%%%%%%%%%%%%%%%%%%%%%%%%%%%%%%%%%%%%%%
%%%%%%%%%%%%%%%%%%%%%%%%%%%%%%%%%%%%%%%%%%%%%%%%%%%%%
%%%%%%%%%%%%%%%%%%%%%%%%%%%%%%%%%%%%%%%%%%%%%%%%%%%%%
%%%%%%%%%%%%%%%%%%%%%%%%%%%%%%%%%%%%%%%%%%%%%%%%%%%%%
%%%%%%%%%%%%%%%%%%%%%%%%%%%%%%%%%%%%%%%%%%%%%%%%%%%%%

\section{Density increment}\label{sec:densityincrement}

In this section, we prove \cref{thm:repeatdensityincrement}, which states that that under certain conditions we can either increase the density of our graph or find a large number of small subgraphs of constant density. We begin with a technical lemma, which, given a graph with average degree close to the minimum degree, allows us to extract a subgraph with maximum degree close to the average degree, while only losing a small amount of density.

\begin{lem}\label{lem:maxdeg}
    For every $\gamma,\alpha,\beta>0$ with $\beta<1$, there exists $\varepsilon=\varepsilon_{\ref{lem:maxdeg}}(\gamma,\alpha,\beta)>0$ such that for any $D\in \N$, if $G$ is a graph with $\delta(G)\geq D$ and $\d(G)\leq (1+\varepsilon)\frac{D}{2}$, then $G$ contains a subgraph $G'$ such that $\d(G')\geq (1-\gamma)\frac{D}{2}$, $\Delta(G')\leq (1+\alpha)D$ and $\v(G')\geq (1-\beta)\v(G)$.
\end{lem}

\begin{proof}
    Choose $\varepsilon=\min\left(\frac{\gamma}{2\left(1+\frac{1}{\alpha}\right)},\alpha\beta\right)$. Let $D,G$ be as in the statement.
    
    Let $X$ be the set of vertices of $G$ of degree greater than $(1+\alpha)D$, and set $G'=G-X$. Clearly $\Delta(G')\leq (1+\alpha)D$. We wish to prove that $\d(G')\geq (1-\gamma)\frac{D}{2}$ and $\v(G')\geq (1-\beta)\v(G)$.
    
    We have that $$\frac{\e(G)}{\v(G)}=\d(G)\leq (1+\varepsilon)\frac{D}{2}.$$ Hence, we have that
    $$(1+\varepsilon)D\v(G)\geq 2\e(G)=\sum_{u\in X}\deg(u)+\sum_{u	\in V(G)\setminus X}\deg(v)\geq \sum_{u\in X}\deg(u)+D(\v(G)-|X|)$$
    or again that
    $$\sum_{u\in X}\deg(u)\leq D(\varepsilon \v(G)+|X|).$$
    By definition of $X$, we then have that
    $(1+\alpha)D|X|\leq D(\varepsilon \v(G)+|X|)$, hence $|X|\leq \frac{\varepsilon}{\alpha} \v(G)$.
    
    This implies that $\v(G')= \v(G)-|X|\geq \left(1-\frac\varepsilon\alpha\right)\v(G)\geq (1-\beta)\v(G)$.
    
    We also get that
    $$\sum_{u\in X} \deg(u)\leq \left(1+\frac{1}{\alpha}\right)D\varepsilon \v(G)$$
    and so
    $$\d(G')=\frac{\e(G')}{\v(G')}\geq \frac{\frac{D}{2}\v(G)-\left(1+\frac{1}{\alpha}\right)D\varepsilon \v(G)}{\v(G)}= \left(1-2\varepsilon\left(1+\frac{1}{\alpha}\right)\right)\frac{D}{2}\geq(1-\gamma)\frac{D}{2}.$$
\end{proof}

The next follows directly from the density increment result of Norin and Song \cite[Theorem 4.1]{norin_breaking_2020} by taking $K=\frac{1}{8\varepsilon}$.

\begin{lem}[{\cite{norin_breaking_2020}}]\label{lem:densityincrement}
   If $0<\eps<\frac{1}{100}$ and $G$ is a graph with $\d(G)\geq\frac{2}{\varepsilon}$, then $G$ contains either
    \begin{enumerate}
        \item a subgraph $J$ such that $\v(J)\leq \frac{\d(G)}{2\varepsilon}$ and $\d(J)\geq \varepsilon \d(G)$, or
        \item a minor $H$ of $G$ such that $\d(H) \geq(1+\varepsilon)\d(G)$.
    \end{enumerate}
\end{lem}

Recursively applying this lemma, we may now deduce the desired result, which we restate for convenience.

\larger*

\begin{proof}
    Take $\varepsilon<\frac{1}{300}$. Let $K$ be given. Set $0<\gamma<1-\frac{1}{(1+3\varepsilon)^{\frac{1}{K+1}}}$ (note that this implies $(1-\gamma)^{K+1}\geq \frac{2}{3}$). We may now choose $C=\frac{K+1+\frac{1}{3\gamma}}{\varepsilon}$ and $\varepsilon'=\min\left(\varepsilon_{\ref{lem:maxdeg}}(\gamma,1,\frac{1}{\varepsilon C}),(1+3\varepsilon)(1-\gamma)^K-1\right)$. Let $D,G$ be as in the statement.
     
   	We will show that either we can construct a minor $H$ of $G$ such that $\d(H)\geq (1+\varepsilon')\frac{D}{2}$ or that we can recursively construct a sequence of subgraphs $G_0,\dots,G_K$ and $J_1,\dots,J_K$ of $G$ such that
    \begin{enumerate}[label=(\alph*)]
    	\item \label{cond:subgraphs} $G_0$ is a subgraph of $G$, and $G_i,J_i$ are disjoint subgraphs of $G_{i-1}$ for $i\in [K]$;
        \item \label{cond:ji} $\v(J_i)\leq \frac{D}{\varepsilon}$ and $\d(J_i)\geq \varepsilon D$ for $i\in [K]$; and
        \item \label{cond:gi} $\v(G_i)\geq \left(C-\frac{i+1}{\varepsilon}\right)D$, $\d(G_i)\geq (1-\gamma)^{i+1}\frac{D}{2}$ for $i\in \{0\}\cup [K]$.
    \end{enumerate}
    
    We first consider the base case of this construction. Either $\d(G)>(1+\varepsilon')\frac{D}{2}$, in which case taking $H=G$ yields the result, or applying Lemma \ref{lem:maxdeg} yields a subgraph $G_0$ of $G$ such that $\d(G_0)\geq (1-\gamma)\frac{D}{2}$, $\Delta(G_0)\leq (1+1)D=2D$ and $\v(G_0)\geq \left(1-\frac{1}{\varepsilon C}\right)\v(G)\geq \left(C-\frac{1}{\varepsilon}\right)D$.
    
    Let $i\in[K]$. Suppose we have already constructed $G_0,\dots,G_{i-1},J_1,\dots,J_{i-1}$. We have that $\d(G_{i-1})\geq(1-\gamma)^{i}\frac{D}{2}\geq \frac{C}{3}>\frac{2}{3\varepsilon}$. Hence, we can apply \cref{lem:densityincrement} (to $G_{i-1}$ with $3\varepsilon)$ to get that $G_{i-1}$ contains either
    \begin{enumerate}
    	\item \label{case:subgraph} a subgraph $J_i$ such that $\v(J_i)\leq \frac{\d(G_{i-1})}{6\varepsilon}$ and $\d(J_i)\geq 3\varepsilon \d(G_{i-1})$, or
    	\item \label{case:minor}a minor $H$ such that $\d(H)\geq (1+3\varepsilon)\d(G_{i-1})$.
    \end{enumerate}
    If we are in case (\ref{case:minor}), then we are done since $$\d(H)\geq(1+3\varepsilon)\d(G_{i-1})\geq (1+3\varepsilon)(1-\gamma)^{i}\frac{D}{2}\geq (1+\varepsilon')\frac{D}{2}.$$
    
    Otherwise we are in case (\ref{case:subgraph}). Set $G_i=G_{i-1}-V(J_i)$. Condition \ref{cond:subgraphs} is direct. Condition \ref{cond:ji} is due to the fact that
    $$\v(J_i)\leq \frac{\d(G_{i-1})}{6\varepsilon}\leq \frac{\Delta(G_0)}{12\varepsilon}\leq\frac{D}{6\varepsilon}<\frac{D}{\varepsilon}$$
    and that $$\d(J_i)\geq 3\varepsilon \d(G_{i-1})\geq 3\varepsilon (1-\gamma)^{i}\frac{D}{2}\geq \varepsilon D.$$
    
    For condition \ref{cond:gi}, we have that
    \begin{align*}
    	\v(G_i)=\v(G_{i-1})-\v(J_i)\geq \left(C-\frac{i}{\varepsilon}\right)D-\frac{D}{\varepsilon}=\left(C-\frac{i+1}{\varepsilon}\right)D
    \end{align*}
    and
    \begin{align*}
    	\d(G_i)
    	&=\frac{\e(G_i)}{\v(G_i)}
    	\geq \frac{\e(G_{i-1})-\Delta(G_0)\cdot \v(J_i)}{\v(G_i)}
    	\geq \frac{\v(G_{i-1})\d(G_{i-1})-2D\cdot\frac{\d(G_{i-1})}{6\varepsilon}}{\v(G_i)}\\
    	&\geq \frac{\v(G_{i})-\frac{D}{3\varepsilon}}{\v(G_i)}\d(G_{i-1})
    	\geq \left(1-\frac{\frac{D}{3\varepsilon}}{\left(C-\frac{i+1}{\varepsilon}\right)D}\right)\d(G_{i-1})\\
    	&\geq (1-\gamma)\d(G_{i-1})
    \end{align*}
    from which it follows that $\d(G_i)\geq (1-\gamma)^{i+1}\frac{D}{2}$.
    
    If the construction was not interrupted (yielding a satisfactory minor $H$), we now have the desired sequence of subgraphs $J_1,\dots,J_K$.
\end{proof}

%%%%%%%%%%%%%%%%%%%%%%%%%%%%%%%%%%%%%%%%%%%%%%%%%%%%%
%%%%%%%%%%%%%%%%%%%%%%%%%%%%%%%%%%%%%%%%%%%%%%%%%%%%%
%%%%%%%%%%%%%%%%%%%%%%%%%%%%%%%%%%%%%%%%%%%%%%%%%%%%%
%%%%%%%%%%%%%%%%%%%%%%%%%%%%%%%%%%%%%%%%%%%%%%%%%%%%%
%%%%%%%%%%%%%%%%%%%%%%%%%%%%%%%%%%%%%%%%%%%%%%%%%%%%%
%%%%%%%%%%%%%%%%%%%%%%%%%%%%%%%%%%%%%%%%%%%%%%%%%%%%%
%%%%%%%%%%%%%%%%%%%%%%%%%%%%%%%%%%%%%%%%%%%%%%%%%%%%%
\section{Building a minor from pieces}\label{sec:minorfrompieces}

In this section, we prove \cref{t:minorfrompieces}, which allows us to build a minor from pieces of sufficient density, using some ideas from the proof of \cite[Theorem 2.6]{norin_breaking_2023}. We first need the following definitions.

Given an injection $\phi : V(H) \to V(G)$ we say that a model  $\mu$  of  $H$ in $G$ is \emph{$\phi$-rooted} if $\phi(v) \in \mu(v)$ for every $v \in V(H)$. Finally, we say that $G$ is \emph{$H$-linked} if $\v(G) \geq \v(H)$  and for every injection  $\phi : V(H) \to V(G)$ there exists a $\phi$-rooted model of $H$ in $G$. Note that every $H$-linked graph has an $H$ minor, but the converse does not hold. 

A \emph{simple edge extension} of a graph $H$ is a graph $H'$ obtained from $H$ by adding a new vertex joined by an edge to at most one vertex of $H$, or two new vertices joined by an edge. In particular, we have $H \subset H'$ and $\e(H')\leq\e(H)+1$. A  \emph{$k$-edge extension} of $H$ is obtained from $H$ by a sequence of at most $k$ simple edge extensions.
If a graph $G$ is $H'$-linked for every $k$-edge extension $H'$ of $H$ then we write that $G$ is 
\emph{$(H+k)$-linked} for brevity, and we say that a graph is \emph{$k$-linked} if it is $(O + k)$-linked, where $O$ is the null graph.

We will use the following version of Menger's theorem \cite{menger_zur_1927}.

\begin{thm}[{\cite{menger_zur_1927}}]\label{thm:menger}
	If $k\in \N$, $G$ is a graph and $U,W\subseteq V(G)$, then either there exists $A,B\subseteq V(G)$ such that $|A\cap B|\leq \ell-1$, $U\subseteq A$ and $W\subseteq V$, or there exists $\ell$ pairwise vertex-disjoint paths each with one end in $A$ and one end in $B$.
\end{thm}

The following lemma is a key element in our proof of \cref{t:minorfrompieces}.

\begin{lem}\label{t:minorfrompiecesold} Let $H$ be a graph  let $F \subseteq E(H)$ be such that $H - F$ is a disjoint union of graphs $H_1,\ldots,H_k$ and such that no edge of $F$ has both ends in the same component of $H-F$. If $G$ is a graph and $G_1,\ldots,G_k$ are pairwise vertex-disjoint subgraphs of $G$ such that
	\begin{itemize}
		\item $G_i$ is $(H_i+|F|)$-linked for every $i \in [k]$, and. 
		\item $\left(G,V(G) \setminus \bigcup_{i=1}^{k}V(G_i)\right)$ is $2|F|$-dense,
	\end{itemize}	  
	then $H\preceq G$. 
\end{lem}

Roughly speaking, we will find find paths between the $G_i$ which correspond to the edges of $F$ by applying Menger's theorem using the second condition of the lemma, after which the first condition will allow us to construct a rooted model of $H_i$ in each $G_i$ while simultaneously rerouting the paths found previously in order to avoid these models, which together yield a model of $H$.

\begin{proof} 	
	For brevity, let $Z = \bigcup_{i=1}^{k}V(G_i)$.

	For every $i\in [k]$, $G_i$ is $|F|$-linked, and so contains at least $2|F|$ vertices. Hence, we can choose $U$ a set of $2|F|$ vertices of $G_1$. Furthermore, it is possible to choose for each $uv\in F$ two vertices $r_{uv}$ and $s_{uv}$ such that $r_{uv}\in V(G_i)$ if $u\in V(H_i)$  and $s_{uv}\in V(G_i)$ if $v\in V(H_i)$ (and such that all these vertices are distinct). Let $W$ be the set of these $2|F|$ vertices. Note that $U$ and $W$ might intersect.
	
	Apply \cref{thm:menger} to $G,U,W$ with $\ell=2|F|$ instead. There first possibility is there exists $A,B\subseteq V(G)$ such that $|A\cap B|\leq k-1$, $U\subseteq A$ and $W\subseteq V$. Since $|A\cap B|\leq 2|F|-1<|U|$, we have that $U\setminus B=U\setminus (A\cap B)\neq \emptyset$, and analogously $W\setminus B\neq \emptyset$. In particular, $A\setminus B,B\setminus A\neq \emptyset$, and so $(A,B)$ is a separation. Given that $\left(G,V(G) \setminus Z\right)$ is $2|F|$-dense, $A\setminus B\subseteq V(G) \setminus Z$. In particular, $(U\setminus B)\cap V(G_1)\subseteq (A\setminus B)\cap V(G_1)=\emptyset$. This is a contradiction since $U\subseteq V(G_1)$ and $U\setminus B\neq \emptyset$.
	
	Hence, for each $r_{uv}\in W$ (resp. $s_{uv}\in W$), there a path $R_{uv}$ (resp. $S_{uv}'$) with one end in $U$, say $r_{uv}'$ (resp. $s_{uv}'$), and the other end is $r_{uv}$ (resp. $s_{uv}$), and all these paths are pairwise vertex-disjoint. By choosing these paths to be as short as possible, we may assume that no internal vertices of the paths are in $G_1$.
	
	As $G_1$ is $|F|$-linked, there exists in $G_1$ pairwise vertex-disjoint paths $T_{uv}$, indexed by $uv\in F$, such that $R_{uv}$ has $r_{uv}'$ and $s_{uv}'$ as ends. For each $uv\in F$, set $P_{uv}$ as the path obtained by concatenating $R_{uv}$, $T_{uv}$ and $S_{uv}$. These are pairwise vertex-disjoint paths in $G$ such that if $u \in V(H_i)$ and $v \in V(H_j)$ then $P_{uv}$ has ends $r_{uv} \in V(G_i)$ and $s_{uv} \in V(G_j)$. We may also assume that $P_{uv}$ paths is shortest possible, and in particular that it has no internal vertices in $G_i\cup G_j$ .
	
	We will need to modify the paths $\{P_{uv}\}_{uv\in F}$ after we find appropriate rooted models in the graphs $G_1,\ldots,G_s$, and we prepare for this as follows.
	
	Fix $uv\in F$. We construct an auxiliary graph $J_{uv}$ with vertex set $[k]$ such that indices $i',i''$ are adjacent if there exists a subpath of $P_{uv}$ with ends in $V(G_{i'})$ and $V(G_{i''})$ and otherwise disjoint from $Z$. Denote such a subpath by $P^{i'i''}_{uv}$. Suppose $u\in V(H_i)$ and $v\in V(H_j)$. Clearly, $i$ and $j$ belong to the same component of $J_{uv}$, and so there exist a sequence $I_{uv}=(i_1,i_2\ldots,i_\ell)$ of distinct indices such that $i_1=i$ and $i_\ell=j$, and $i_t$ and $i_{t+1}$ are adjacent in our auxiliary graph for each $1 \leq t \leq \ell$. For each $1 \leq t \leq \ell$, we define $s^{i_t}_{uv}$ and $r^{i_t}_{uv}$ to be the ends of $P^{i_{t-1}i_t}_{uv}$ and $P^{i_{t}i_{t+1}}_{uv}$, respectively, in $V(G_{i_t})$ (except for $s_{uv}^{i_1},r_{uv}^{i_\ell}$ which are not defined). It is possible that $r^{i_t}_{uv}=s^{i_t}_{uv}$. Note that $r^{i_1}_{uv}=r_{uv}$ and $s^{i_{\ell}}_{uv}=s_{uv}$.
	
 	For each $i \in [k]$, let $H'_i$ be the $|F|$-extension of $H_i$ defined as follows. For each edge $uv \in F$ such that $u,v\notin V(H_i)$, if $r^{i}_{uv}$ and $s^{i}_{uv}$ are both defined and distinct, we add $p^{i}_{uv}$ and $q^{i}_{uv}$ to $H_i$ joined by an edge, and if $r^{i}_{uv}=s^{i}_{uv}$ is defined, we add an isolated vertex $p^{i}_{uv}=q^{i}_{uv}$. Furthermore, if $u \in V(H_i)$, add a vertex $p_{uv}^i$ to $H_i$ and an edge from $u$ to $p_{uv}^i$, and if $v\in V(H_i)$, add a vertex $q_{uv}^i$ to $H_i$ and an edge from $v$ to $q_{uv}^i$.
	
	Let an injection $\phi_i: V(H'_i) \to V(G_i)$ be defined as follows. For $uv\in F$, let $\phi_i(p^{i}_{uv})=r^{i}_{uv}$ and $\phi_i(q^{i}_{uv})=s^{i}_{uv}$ whenever $r^{i}_{uv}$ and $s^{i}_{uv}$, respectively, are defined. For $w \in V(H_i)$, we choose $\phi_i(w)$ arbitrarily subject to $\phi_i$ being injective (this is possible given than $G_i$ is $(H_i+|F|)$-linked and so contains at least $\v(H_i)+2|F|$ vertices).
	
	As $G_i$ is $(H_i+|F|)$-linked there exists a $\phi_i$-rooted model $\mu_i$ of $H'_i$ in $G_i$. Let $\mu'$ be a model of $H - F$ in $G$ obtained by defining $\mu'(v)=\mu_i(v)$ when $v \in V(H_i)$.
	 
	To extend $\mu'$ to a model $\mu$ of $H$ it suffices to find a set of pairwise internally vertex-disjoint paths $\mc{Q}=\{Q_{uv}\}_{uv \in F}$ such that for each $uv \in F$ the path $Q_{uv}$ has one end in $\mu'(u)$ and the other in $\mu'(v)$, and otherwise disjoint from $\mu'(V(H))$. To do this, we consider the subgraph $G_{uv}$ of $G$ induced by the set
	$$W_{uv} = (V(P_{uv}) \setminus Z)  \cup \mu'(p_{uv}^{i_1}) \cup \mu'(q_{uv}^{i_\ell}) \cup \bigcup_{i \in I_{uv} \,:\, u,v \not \in V(H_i)}\left(\mu'(p^i_{uv}) \cup \mu'(q^i_{uv})\right),$$
	 Note that the elements of $\{W_{uv}\}_{uv \in F}$ are pairwise vertex-disjoint and so it suffices to show that for each $uv \in F$ there exists a path $Q_{uv}$ with one end in $\mu'(u)$, the other in $\mu'(v)$ and all internal vertices in $W_{uv}$. As $W_{uv}$ contains vertices $x_{uv} \in \mu'(p_{uv}^{i_1})$ with a neighbour in  $\mu'(u)$ and $y_{uv} \in \mu'(q_{uv}^{i_\ell})$ with a neighbour in $\mu'(v)$, it suffices to show that $x_{uv}$ and $y_{uv}$ belong to the same component of $G_{uv}$. This follows from our construction. Indeed, if $I_{uv}=(i_1,i_2\ldots,i_\ell)$, then $G_{uv}$ contains
	\begin{itemize}
		\item a path from $x_{uv}$ to $r_{uv}=r_{uv}^1$ in $\mu'(p_{uv}^{i_1})$,
		\item a path $P^{i_1i_2}_{uv}$ from $r_{uv}^{i_1}$ to $s^{i_2}_{uv}$,
		\item a path from $s^{i_2}_{uv}$ to $r^{i_2}_{uv}$ in $\mu'(q^{i_2}_{uv}) \cup \mu'(p^{i_2}_{uv})$,
		\item similarly paths from $r^{i_{t-1}}_{uv}$ to $s^{i_{t}}_{uv}$ and from $s^{i_{t}}_{uv}$ to $r^{i_{t}}_{uv}$ for $t=3,\ldots,\ell$, 
		\item a path from $s_{uv}^{\ell}=s_{uv}$ to $y_{uv}$ in $\mu'(q_{uv}^{i_{\ell}})$.
	\end{itemize}
	This completes the proof of the lemma.
\end{proof}

In order to prove \cref{t:minorfrompieces} using \cref{t:minorfrompiecesold}, we need to find well-linked pieces of $G$ using dense pieces of $G$.

We define the \emph{connectivity} of $G$, denoted by $\kappa(G)$, as the minimum order of a separation of $G$, except when $G$ is complete in which case $\kappa(G)=\v(G)-1$. It is easy to see that the connectivity of a graph gives a lower bound on the degrees of the vertices of the graph, that is $\d(G)\geq \frac{\kappa(G)}{2}$. The following result of Mader \cite{mader_existenz_1972} allows us to find a subgraph of connectivity on the same order as the average degree.

\begin{thm}[{\cite{mader_existenz_1972}}]\label{thm:maderconnectivitydensity}
	Every graph $G$ contains a subgraph $G'$ such that $\kappa(G')\geq \frac{\d(G)}{2}$.
\end{thm}

We need the following result of Wollan \cite{wollan_extremal_2008}.

\begin{thm}[{\cite[Theorem 1.1]{wollan_extremal_2008}}]\label{t:Hlinked}
If $H$ and $G$ are graphs such that $\kappa(G) \geq \v(H)$ and $\d(H) \geq 9c(H) + 26833\v(H)$, then $G$ is $H$-linked.
\end{thm}

The following form will be easier to use.

\begin{cor}\label{cor:Hlinked}
	There exists $C=C_{\ref{cor:Hlinked}}\in \N$ such that if $H$ and $G$ are graphs such that $\kappa(G)\geq Cc(H)$, then $G$ is $H$-linked.
\end{cor}

\begin{proof}
	Set $C=6\cdot 28633$. Let $H,G$ be as in the statement. If $\v(H)\leq 2$, the statement is trivial. Hence we may suppose that $\v(H)\geq 3$.
	
	First note that $c(H)>\frac{\v(H)}{2}-1$ for every graph $H$. Indeed, the complete graph $K_{\v(H)-1}$ on $\v(H)-1$ vertices does not contain $\v(H)$ as a minor, but $\d(K_{\v(H)-1})=\frac{\binom{\v(H)-1}{2}}{\v(H)-1}=\frac{\v(H)}{2}-1$. In particular, $Cc(H)>C\left(\frac{\v(H)}{2}-1\right)\geq 2\cdot 26833\v(H)$, using that $\v(H)\geq 3$.
	
	Then,
	$$\kappa(G)\geq Cc(H)\geq \v(H)$$
	and
	$$\d(G)\geq \frac{\delta(G)}{2}\geq \frac{\kappa(G)}{2}\geq \frac{Cc(H)}{2}\geq 9c(H)+26833\v(H).$$
	We may therefore apply \cref{t:Hlinked} to obtain that $G$ is $H$-linked, which completes the proof.
\end{proof}

We may now upper bound the extremal function for graphs with edge extensions.

\begin{lem}\label{lem:extensionextremal}
	If $J$ is a $k$-edge extension of a graph $H$, then $c(J)\leq 2C_{\ref{cor:Hlinked}}c(H)+4k$.
\end{lem}

\begin{proof}
	Suppose $G$ is such that $\d(G)>2C_{\ref{cor:Hlinked}}c(H)+4k$. We wish to prove that $H$ is a minor of $G$.
	
	By \cref{thm:maderconnectivitydensity}, $G$ contains a subgraph $G'$ such that $\kappa(G')\geq \frac{\d(G)}{2}\geq C_{\ref{cor:Hlinked}}c(H)+2k$.
	
	As $J$ is a $k$-edge extension of $H$, we may consider $H$ as an induced subgraph of $J$. Let $A$ be the the vertices added by simple edge extensions. In particular, $|A|\leq 2k$ and $V(J)=V(H)\cup A$. Let $B$ be the vertices of $V(H)$ to which incident edges have been added by (the first type of) simple edge extensions. Note that by definition of simple edge extensions, $J'=J[A\cup B]-E(J[B])$ is necessarily a forest.
	
	Given that $\delta(G')\geq\kappa(G')\geq 2k$, we may greedily construct an embedding of $J'$ in $G'$ : consecutively for each component (which is a tree) of $J'$, select an arbitrary unused vertex of $G'$ and repeatedly use the minimum degree to find new vertices in order construct the tree. Let $\phi:V(J')\rightarrow V(G')$ be the (injective) mapping of vertices in this embedding.
	
	Given that $\kappa(G'-\phi(A))\geq \kappa(G')-2k\geq C_{\ref{cor:Hlinked}}c(H)$, there exists by \cref{cor:Hlinked} a $\phi|_B$-rooted model $\mu$ of $H$ in $G'-\phi(A)$. We can then expand $\mu$ to a model of $J$ in $G'$ by setting $\mu(a)=\{\phi(a)\}$ for every $a\in A$. Hence, $J\preceq G'\preceq G$ as desired, which completes the proof.
\end{proof}

Finally, we prove \cref{t:minorfrompieces}, which we restate for convenience.

\minorfrompieces*

\begin{proof}
	We show that the statement holds for $C=8C_{\ref{cor:Hlinked}}^2$.
	
	By \cref{thm:maderconnectivitydensity}, for each $i\in [k]$ there exists a subgraph $G_i'$ of $G_i$ such that $\kappa(G_i')\geq \frac{\d(G_i)}{2}.$
	
	We first claim that $G_i'$ is $(H_i+|F|)$-linked for every $i\in [k]$. We thus need to show that if $J$ is an $|F|$-edge extension of $H_i$, then $G_i'$ is $J$-linked. By \cref{lem:extensionextremal}, we have that
	$$\kappa(G_i')\geq \frac{\d(G_i)}{2}\geq 4C_{\ref{cor:Hlinked}}^2(c(H_i)+|F|)\geq C_{\ref{cor:Hlinked}}(2C_{\ref{cor:Hlinked}}c(H_i)+4|F|)\geq C_{\ref{cor:Hlinked}}c(J).$$
	By \cref{cor:Hlinked}, $G_i'$ is $J$-linked, which completes the proof of the claim.
	
	Since $V(G) \setminus \bigcup_{i=1}^{k}V(G_i) \subseteq V(G) \setminus \bigcup_{i=1}^{k}V(G_i')$, we have that $\left(G,V(G) \setminus \bigcup_{i=1}^{k}V(G_i')\right)$ is also $2|F|$-dense.
	
	Applying \cref{t:minorfrompiecesold} with $G_i'$ instead of $G_i$, we have that $H\preceq G$ as desired.
\end{proof}

%%%%%%%%%%%%%%%%%%%%%%%%%%%%%%%%%%%%%%%%%%%%%%%%%%%%%
%%%%%%%%%%%%%%%%%%%%%%%%%%%%%%%%%%%%%%%%%%%%%%%%%%%%%
%%%%%%%%%%%%%%%%%%%%%%%%%%%%%%%%%%%%%%%%%%%%%%%%%%%%%
%%%%%%%%%%%%%%%%%%%%%%%%%%%%%%%%%%%%%%%%%%%%%%%%%%%%%
%%%%%%%%%%%%%%%%%%%%%%%%%%%%%%%%%%%%%%%%%%%%%%%%%%%%%
%%%%%%%%%%%%%%%%%%%%%%%%%%%%%%%%%%%%%%%%%%%%%%%%%%%%%
%%%%%%%%%%%%%%%%%%%%%%%%%%%%%%%%%%%%%%%%%%%%%%%%%%%%%

\section{Tightness}\label{sec:tightness}

In this section, we show the necessity of the conditions imposed in \cref{thm:main}. Indeed, for each condition we provide examples showing that the theorem does not hold when removing this condition but maintaining all others.

In fact, in all of these examples, we will even impose a lower bound on the minimum (or average, in the last case) degree of $G$ which is larger than $\v(H)-1$ as in \cref{thm:main}.

\subsection*{Removing the maximum degree bound on $H$.}

\begin{lem}
	For $s, t\in \N$ such that $s\ll t$,  there exists a graph $G$ with $\delta(G) \geq 2s+t-2\sqrt{2s}-2$ and no $K_{s,t}$ minor.
\end{lem}

\begin{proof}
	Let $d=\lceil\sqrt{2s}\rceil$ and let $G'$ be a $d$-regular graph of girth (length of shortest cycle) greater than $(d+1)s$ with $$\v(G')=2\s{s+ \left\lfloor\frac{t-\sqrt{2s}}{2}\right\rfloor}.$$
	
	Note that it is well known that $n$ vertex $d$-regular graphs of girth at least $g$ exist for all even $n$ large enough compared to $d$ and $g$. In particular, by a result of Sauer~\cite{sauer_extremaleigneschaften_1967-1,sauer_extremaleigneschaften_1967} (see \cite[Theorem 7]{exoo_dynamic_2013}) they exist for even $n \geq 2(d-1)^{g-2}$ for $d,g \geq 3$. Thus the graph $G'$ as above exists for $t \gg s$. 
	
	We show that the complement $G$ of $G'$ satisfies the lemma. Let $k=\v(G')-s-t$. As
	$$\delta(G)
	= \v(G)-d-1
	=2\s{s+ \left\lfloor\frac{t-\sqrt{2s}}{2}\right\rfloor} - \lceil\sqrt{2s}\rceil -1
	\geq 2s+t-2\sqrt{2s}-2,$$ 
	it remains to show $G$ has no $K_{s,t}$ minor. 

	Suppose for a contradiction that there exists a model $\mu$ of $K_{s,t}$ in $G$. Let $A$ and $B$ denote the parts of the bipartition of $K_{s,t}$ of sizes $s$ and $t$, respectively. Let $A'$ denote the set of vertices $v$ in $A$ such that $|\mu(v)|=1$ and  let $A''=\bigcup_{v \in A'} \mu(v)$. Symmetrically, let $B'$ denote the set of vertices $v$ in $B$  such that $|\mu(v)|=1$  and  let $B''=\bigcup_{v \in B'} \mu(v)$. 
	Then $$\v(G) \geq \sum_{v \in A \cup B} |\mu(v)| \geq 2(s+t) - |A''|-|B''|.$$
	As $k=\v(G)-s-t$ and $|B''|\leq t$, it follows that \begin{equation}\label{e:ABlow}
	|A''|+|B''| \geq s+t-k \qquad \mathrm{and} \qquad |A''| \geq s-k.
	\end{equation}
	Let $F = G'[A'' \cup N_{G'}(A'')]$. As every vertex in $A''$ is adjacent in $G$ to every vertex of $B''$ we have $N_{G'}(A'') \cap B'' = \emptyset$. Thus \begin{equation}\label{e:vFlow}
	\v(F) \leq \v(G) - |B''|.
	\end{equation}
	Moreover, $\v(F) \leq (d+1)|A''| \leq (d+1)s$ and so $F$ is a forest (since any cycle in $G'$ contains more than $(d+1)s$ vertices). As $E(F)$ contains all edges of $G'$ incident to vertices of $A''$, and at most $|A''|-1$ edges have both ends in $A''$, we have   \begin{equation}\label{e:vFhigh}\v(F)-1 \geq \e(F) \geq d|A''|- \left(|A''|-1\right).\end{equation}
	Combining \eqref{e:vFlow} and \eqref{e:vFhigh} we have $\v(G) \geq (d-1)|A''|+|B''|+2$, which by \eqref{e:ABlow} implies $s+t+k=\v(G)\geq   s+t-k + (d-2)(s-k)+2$ and so $k\geq\frac{(d-2)s+2}{d}$. It follows that
	$$\v(G')= s+t+k\geq s+t+\frac{(d-2)s+2}{d}> 2s+t-\frac{2s}{d} \geq 2s + t - \sqrt{2s}> 2\s{s+ \left\lfloor\frac{t-\sqrt{2s}}{2}\right\rfloor},$$
	which contradicts the choice of $\v(G')$.
\end{proof}

\subsection*{Removing the condition that $H$ is bipartite.}

\begin{lem}
	For every $s\in \N$, there exists a graph $G$ with $\delta(G)\geq \frac{22}{3}s-2$ and no $sK_{7}$ minor.
\end{lem}

\begin{proof} Let $t=\left\lfloor\frac{11s-1}{3}\right\rfloor$ and let $G$ be a complete $3$-partite graph with each part of size $t$. Then $\delta(G)= 2t \geq \frac{22}{3}s -2$ and so it only remains to show that $G$ has no $sK_7$ minor. We show that for any model $\mu$ of $K_7$ in $G$ we have $\sum_{v \in V(K_7)}|\mu(v)| \geq 11$. This would prove the statement, since this would imply that any model of $sK_7$ in $G$ would require at least $11s>3t=\v(G)$ vertices.
	
Suppose that $\mu$ is a model of $K_7$ in $G$ with  $\sum_{v \in V(K_7)}|\mu(v)| \leq 10$ and let $A = \{v \in V(K_7) : |\mu(v)|=1\}$. Then $|A| \geq 4$, and by the pigeonhole principle there exist $v,v' \in A$ such that the unique elements of $\mu(v)$ and $\mu(v')$ belong to the same part of the tripartition of $G$. This is a contradiction, as $v,v'\in E(K_7)$ but there is no edge in $G$ between the elements of $\mu(v)$ and $\mu(v')$.
\end{proof}	

\subsection*{Removing the condition that $H$ has good separation properties.}\hfill\\

We first need the following theorem, which is a direct consequence of a result of Norin, Reed, Thomason and Wood \cite[Theorem 4]{norin_lower_2020}.

\begin{thm}[{\cite{norin_lower_2020}}]\label{t:nosss}
There exists $d_0\in \N$ such that for every integer $d \geq d_0$ and every integer $n \geq 2d+1$, there exist a graph $H$ with $\v(H) = n$ and $\e(H)=dn$ and a graph $G$ with  $\d(G) \geq \frac{1}{4} n \sqrt{\ln d}$ such that $H$ is not a minor of $G$.
\end{thm}

We wish to use \cref{t:nosss} to construct an example showing that \cref{thm:main} does not hold if we remove that $H$ is in a family with strongly sublinear separators. We will need the following lemma in order for $H$ to be bipartite and have bounded maximum degree.

\begin{lem}\label{l:bipdelta} For every $\Delta \in \N^{\geq 3}$ and every graph $H$, there exists a bipartite graph $H'$ with $\Delta(H') \leq \Delta$ and $\v(H') \leq \frac{4\e(H)}{\Delta-2}+2\v(H)$ such that $H$ is a minor of $H'$. 
\end{lem}

\begin{proof}

	For each $v \in V(H)$, let $k(v) = \left\lfloor \frac{\deg_H(v)}{\Delta-2}+1\right\rfloor$, and let $A(v)$ and $B(v)$ be sets with $|A(v)|=|B(v)|=k(v)$ and all the sets in $\{A(v),B(v)\}_{v \in V(H)}$ are pairwise disjoint.
	
	Let $A=\bigcup_{v \in V(H)}A(v)$, $B=\bigcup_{v \in V(H)}B(v)$. We construct $H'$ so that $(A,B)$ is a bipartition of it; in particular, $V(H') = A \cup B$. Note that 
	$$\v(H') = 2\sum_{v \in V(H)}k(v) \leq  2\sum_{v \in V(H)}\s{\frac{\deg_H(v)}{\Delta-2}+1} = \frac{4\e(H)}{\Delta-2}+2\v(H),$$
	as desired.
	
	We now define the edge set of $H'$. For each $v \in V(H)$, let $H'[A(v) \cup B(v)]$ be a path (chosen arbitrarily) respecting the bipartition $(A,B)$, which we denote by $P(v)$. For every edge $uv \in E(H)$, add an edge to $H'$ with one end in $A(u)$ and another in $B(v)$. Note that by the choice of $k(v)$ it is possible to add such edges so that $\Delta(H') \leq \Delta$. Contracting each path $P(v)$ to a single vertex we obtain $H$ as a minor of $H'$.
\end{proof}

\begin{cor}\label{cor:nosss}
	There exist $\Delta_0\in \N$ such that for every integer $\Delta \geq \Delta_0$ and every integer $n \geq 2\Delta$, there exist a bipartite graph $H'$ with  $\v(H') \leq 6n$ and $\Delta(H') \leq \Delta$ and a graph $G'$ with $\delta(G') \geq \frac{1}{5} n \sqrt{\ln \Delta}$ such that $H'$  is not a minor of $G'$.
\end{cor}

\begin{proof}
	Let $d_0$ be as in \cref{t:nosss}. We show that $\Delta_0=\max\{5,d_0+2\}$ satisfies the corollary.
	For $\Delta \geq \Delta_0$  and $n \geq 2\Delta$, let $d=\Delta -2$. Then $d \geq d_0$ and $n \geq 2d+1$ so by \cref{t:nosss} there exists a graph $H$ with $\v(H) = n$ and $\e(H)=dn$ and a graph $G$ with $\d(G) \geq \frac{1}{4} n \sqrt{\ln d}$  such that $H$ is not a minor of $G$.
	
	Let $G'$ be a subgraph of $G$ with $\delta(G') \geq \d(G)$ (it is a standard result that this is always possible by repeatedly removing vertices of smaller degree, see for instance \cite[Proposition 1.2.2]{diestel_graph_2017}). Then $$\delta(G') \geq \frac{1}{4} n \sqrt{\ln (\Delta-2)} \geq  \frac{1}{5} n \sqrt{\ln (\Delta)},$$
as $\sqrt{\ln (\Delta-2)} \geq \frac{4}{5} \sqrt{\ln (\Delta)}$ for $\Delta \geq 5$. By \cref{l:bipdelta} there exists  bipartite graph $H'$ with $\Delta(H') \leq \Delta$ and $\v(H') \leq \frac{4\e(H)}{\Delta-2}+2\v(H) = 6\v(H)$ such that $H$ is a minor of $H'$.

As $H$ is a minor of $H'$, $G'$ is a minor of $G$ and  $H$ is not a minor of $G$, it follows that $H'$ is not a minor of $G'$, and so $H'$ and $G'$ are as desired.
\end{proof}

\cref{cor:nosss} provides interesting examples for our purposes when $\Delta\geq e^{900}$.

\subsection*{Replacing the minimum degree of $G$ with average degree.}

\begin{lem} 
	For all $s,t\in \N$ and $\eps >0$, there exists a graph $G$ with $\d(G) \geq st-1+\frac{t-1}{2}-\eps$ and no $sK_{t,t}$ minor. 
\end{lem}

\begin{proof}
	Let $k$ be a positive integer, and let  $G=G(s,t,k)$ be constructed as follows. Let $X_1,\ldots, X_k,Y$ be disjoint sets such that $|X_i|=t$ for $i=[k]$ and  $|Y|=st-1$. Let $V(G)=X_1 \cup \ldots \cup X_k \cup Y$ and let distinct $u,v \in V(G)$ be adjacent unless $u \in X_i$ and $v \in X_j$ for some $i \neq j$. In other words, $X_i$ is a clique of order $t$ for every $i\in [k]$ and $Y$ is a set of $st-1$ universal vertices. First, we have that $\v(G)=kt+st-1$. Then
	$$\e(G)=\binom{st-1}{2}+(st-1)\cdot kt+k\binom{t}{2}.$$
	
	One easily computes that when $k\rightarrow \infty$, we have $$\frac{\binom{st-1}{2}}{kt+st-1}\rightarrow 0,\text{ } \frac{(st-1)\cdot kt}{kt+st-1}\rightarrow st-1 \text{ and } \frac{\binom{t}{2}}{kt+st-1}\rightarrow \frac{t-1}{2}.$$ Hence, $\d(G)\geq st-1+\frac{t-1}{2}-\varepsilon$ for sufficiently large $k$.
	
	We will show  that for every model $\mu$ of $K_{t,t}$ in $G$ we have $|Y \cap (\bigcup_{v \in V(K_{t,t})}\mu(v)) | \geq t$. Given that $|Y|$ has size $st-1$, this will imply that $sK_{t,t}$ is not a minor of $G$ and thus prove the lemma.
	
	Suppose for a contradiction that $\mu$ is a model of $K_{t,t}$ in $G$ and $\left|Y \cap \left(\bigcup_{v \in V(K_{t,t})}\mu(v)\right) \right| \leq t-1$. Let $U = \{v \in V(K_{t,t}) : \mu(v) \cap Y =\emptyset\}$. Then $|U| \geq t+1$ and so $K_{t,t}[U]$ is connected. It follows that $G\left[\bigcup_{v\in U} \mu(v)\right]$ is connected, but $\bigcup_{v\in U} \mu(v) \subseteq V(G)-Y$ and so $\bigcup_{v\in U} \mu(v) \subseteq X_i$ for some $1 \leq i \leq k$. Hence, $t \geq\left|\bigcup_{v\in U} \mu(v) \right|\geq |U| \geq t+1$, a contradiction.  
\end{proof}

%%%%%%%%%%%%%%%%%%%%%%%%%%%%%%%%%%%%%%%%%%%%%%%%%%%%%
%%%%%%%%%%%%%%%%%%%%%%%%%%%%%%%%%%%%%%%%%%%%%%%%%%%%%
%%%%%%%%%%%%%%%%%%%%%%%%%%%%%%%%%%%%%%%%%%%%%%%%%%%%%
%%%%%%%%%%%%%%%%%%%%%%%%%%%%%%%%%%%%%%%%%%%%%%%%%%%%%
%%%%%%%%%%%%%%%%%%%%%%%%%%%%%%%%%%%%%%%%%%%%%%%%%%%%%
%%%%%%%%%%%%%%%%%%%%%%%%%%%%%%%%%%%%%%%%%%%%%%%%%%%%%
%%%%%%%%%%%%%%%%%%%%%%%%%%%%%%%%%%%%%%%%%%%%%%%%%%%%%

\section*{Acknowledgments}
We are grateful to Raphael Steiner for his question about possible strengthenings of an earlier version of \cref{thm:main} which led us to the current formulation.
 
\bibliography{refs}
\bibliographystyle{abbrvurl}

%%%%%%%%%%%%%%%%%%%%%%%%%%%%%%%%%%%%%%%%%%%%%%%%%%%%%
%%%%%%%%%%%%%%%%%%%%%%%%%%%%%%%%%%%%%%%%%%%%%%%%%%%%%
%%%%%%%%%%%%%%%%%%%%%%%%%%%%%%%%%%%%%%%%%%%%%%%%%%%%%
%%%%%%%%%%%%%%%%%%%%%%%%%%%%%%%%%%%%%%%%%%%%%%%%%%%%%
%%%%%%%%%%%%%%%%%%%%%%%%%%%%%%%%%%%%%%%%%%%%%%%%%%%%%
%%%%%%%%%%%%%%%%%%%%%%%%%%%%%%%%%%%%%%%%%%%%%%%%%%%%%
%%%%%%%%%%%%%%%%%%%%%%%%%%%%%%%%%%%%%%%%%%%%%%%%%%%%%

\appendix

\section{Proof of \cref{lem:compsize}}\label{app:compsize}

A \emph{decomposition} of a graph $H$ is a set $\mathcal B\subseteq \mathcal P\left(V(H)\right)$ such that for every edge $uv\in E(H)$, there exists $B\in\mathcal B$ such that $u,v\in B$. For $C\in \N$, we say $\mathcal B$ is \emph{$C$-bounded} if $|B|\leq C$ for all $B\in \mathcal B$. The \emph{excess} of $\mathcal B$ is the quantity $\sum_{B\in \mathcal B} |B|-\v(G). $

We need the following well-known lemma \cite{eppstein_densities_2010}, using the formulation by Hendrey, Norin and Wood \cite{hendrey_extremal_2022}.

\begin{lem}[{\cite[Lemma 7.1]{hendrey_extremal_2022}}]\label{lem:boundedecomp}
	For every graph family $\mathcal F$ with strongly sublinear separators and every $\varepsilon > 0$ there exists $C = C_{\ref{lem:boundedecomp}}(\mathcal F,\varepsilon)$ such that every graph $H \in \mathcal F$ admits a $C$-bounded decomposition with excess at most $\varepsilon \v(G)$.
\end{lem}

We can then derive \cref{lem:compsize}, which we restate for convenience. We recall that this is essentially \cite[Lemma 7.3]{hendrey_extremal_2022}.

\compsize*

\begin{proof}
	Take $s=C_{\ref{lem:boundedecomp}}(\mathcal F,\delta)$ and let $\mathcal B$ be a $s$-bounded decomposition of $H$ such that $\sum_{B\in \mathcal B}|B|\leq \left(1+\delta\right)\v(H)$, which exists by Lemma \ref{lem:boundedecomp}.
	
	We first prove the statement when $H$ does not contain any isolated vertices.
	
	Say $|\mathcal B|=k$. Then, write $H_1,\dots,H_k$ as the subgraphs of $H$ induced respectively by the vertex sets in $\mathcal B$. From now on, we consider the vertex sets of $H_1,\dots,H_k$ to all be mutually disjoint. For each $x\in \v(H)$, let $V_x=\{y_1,\dots,y_{i_x}\}$ (for some $i_x\in \N$) be the set of vertices respectively in $H_1,\dots,H_k$ which are copies of $x$. Then, let $F_x=\{y_1y_2,\dots,y_{i_x-1}y_{i_x}\}$, $F=\bigcup_{x\in V(H)}F_x$ and finally $H'=H_1\cup\dots\cup H_k+F$.
	
	Contracting the edges in $F$ shows that $H$ is a minor of $H'$, which is possible since the copies of $x\in V(H)$ form a path in $H'$ and all other edges of $H$ are contained in at least one of $H_1,\dots,H_k$ by the definition of a decomposition. Furthermore, 
	$\v(H')=\sum_{i=1}^k\v(H_i)=\sum_{B\in \mathcal B}|B|\leq (1+\delta)\v(H)$. Since every vertex of $H'$ gains at most 2 edges in $H'$ compared to the original vertex in $H$ (the possible extra edges being those in $F$), $\Delta(H')\leq \Delta(H)+2$. By our construction, when a vertex $x\in V(H)$ has multiplicity $i_x$ in $\mathcal B$, $i_x-1$ edges are added to $F$. Since the total multiplicity of all vertices is $\v(H')$, then $|F|=\v(H')-\v(H)\leq \delta\v(H)$. By construction the components of $H'-F$ are $H_1,\dots,H_k$, which are induced subgraphs of $H$ and have order at most $s$. Finally, given that edges of $F$ are between copies of a vertex in different $H_1,\dots,H_k$, no edge of $F$ has both ends in the same component of $H'-F$. 
	
	To prove the general statement, let $H_0$ be the graph $H$ with all isolated vertices removed. Let $H_0'$ be the graph resulting from applying the statement to $H_0$ (which is possible given that $\mathcal F$ is closed under taking subgraphs), and let $H'$ be the disjoint union of $H_0'$ and of $\v(H)-\v(H_0)$ isolated vertices. It is easy to verify that $H'$ respects the conditions of the statement.
\end{proof}

\end{document}